\documentclass[a4paper,11pt]{article}

\usepackage{amsmath,amsthm,amssymb,amsfonts}
\usepackage{mathtools} 
\usepackage{bm}
\usepackage{mathrsfs} 
\usepackage{changes} 
\usepackage{url} 
\usepackage[toc,page]{appendix} 
\usepackage{enumerate}

\usepackage{graphicx} 
\usepackage{longtable} 
\usepackage[all]{xy} 
\usepackage{newfloat} 
\DeclareFloatingEnvironment[fileext=frm,placement={ht},name=Frame]{algfloat}
\usepackage{caption}
\usepackage{placeins}
\usepackage{tikz}
\usepackage{subfig}
\usepackage{pgfplots}
\usepackage{color}
\usepackage{xcolor} 
\usepackage{array}
\usepackage{booktabs}
\usepackage{grffile}

\newcommand{\R}{\mathbb{R}}

\newcommand{\inner}[2]{\langle #1\,|\,#2\rangle} 

\newcommand{\norm}[1]{\|{#1}\|}
\newcommand{\normq}[1]{\|{#1}\|^2}

\newcommand{\HH}{{\mathcal{H}}} 

\newcommand{\tos}{\rightrightarrows} 
\newcommand{\wto}{\rightharpoonup}

\newcommand{\comenta}[1]{} 

\newcommand{\lab}[1]{\label{#1}}

\newcommand{\bprop}{\begin{proposition}}
\newcommand{\eprop}{\end{proposition}}
\newcommand{\blemm}{\begin{lemma}}
\newcommand{\elemm}{\end{lemma}}
\newcommand{\bdefi}{\begin{definition}}
\newcommand{\edefi}{\end{definition}}
\newcommand{\btheo}{\begin{theorem}}
\newcommand{\etheo}{\end{theorem}}
\newcommand{\bproo}{\begin{proof}}
\newcommand{\eproo}{\end{proof}}
\newcommand{\brema}{\begin{remark}}
\newcommand{\erema}{\end{remark}}
\newcommand{\bitem}{\begin{itemize}}
\newcommand{\eitem}{\end{itemize}}
\newcommand{\bexam}{\begin{example}}
\newcommand{\eexam}{\end{example}}
\newcommand{\bassu}{\begin{assumption}}
\newcommand{\eassu}{\end{assumption}}
\newcommand{\bcoro}{\begin{corollary}}
\newcommand{\ecoro}{\end{corollary}}

\newcommand{\mgap}{\vspace{.1in}}

\newcommand{\beq}{\begin{equation}}
\newcommand{\eeq}{\end{equation}}

\newtheorem{theorem}{Theorem}[section] 
\newtheorem{lemma}[theorem]{Lemma}
\newtheorem{corollary}[theorem]{Corollary}
\newtheorem{proposition}[theorem]{Proposition}
\newtheorem{remark}[theorem]{Remark}
\newtheorem{definition}[theorem]{Definition}
\newtheorem{assumption}[theorem]{Assumption}
\newtheorem{example}[theorem]{Example}

\usepackage[linesnumbered,ruled]{algorithm2e}
\SetKwInput{KwInput}{Input}
\SetKwInput{KwOutput}{Global Variables for Function}
\SetKw{KwGlob}{Global Variables for Function:}
\SetKwFunction{FMain}{BrackBisecProc} 
\SetKwProg{Fn}{Function}{:}{}
\usepackage[T1]{fontenc}
\usepackage[utf8]{inputenc}
\usepackage[portuguese,english]{babel} 

\usepackage{epstopdf}
\newcommand{\HHH}{\boldsymbol{\HH}}

\newcommand{\sumn}{{\sum_{i=1}^n\,}}
\makeatletter
\def\widebreve#1{\mathop{\vbox{\m@th\ialign{##\crcr\noalign{\kern\p@}%
  \brevefill\crcr\noalign{\kern0.1\p@\nointerlineskip}%
  $\hfil\displaystyle{#1}\hfil$\crcr}}}\limits}

\def\brevefill{$\m@th \setbox\z@\hbox{}%
 \hfill\scalebox{0.7}{\rotatebox[origin=c]{90}{(}} \kern4pt $}
\makeatletter

\begin{document}

\title{Projective splitting with backward, half-forward and proximal-Newton steps}
\author{
    M. Marques Alves
\thanks{
Departamento de Matem\'atica,
Universidade Federal de Santa Catarina,
Florian\'opolis, Brazil, 88040-900 ({\tt maicon.alves@ufsc.br}).
This work was partially supported by CNPq grant 308036/2021-2.}
}

\maketitle

\begin{abstract}
We propose and study the weak convergence of a projective splitting 
algorithm~\cite{combettes2014Coupled, eck.sva-fam.mp08} for solving multi-term composite monotone inclusion problems involving the finite sum of $n$ maximal monotone operators, each of which having an inner four-block structure: sum of maximal monotone, Lipschitz continuous, cocoercive and smooth differentiable operators. 
We show how to perform backward and half-forward~\cite{davis2018half} steps with respect to the maximal monotone and 
Lipschitz$+$cocoercive components, respectively, while performing proximal-Newton steps with respect to smooth differentiable  
blocks.
\end{abstract}

\pagestyle{plain}


\section{Introduction}
 \label{sec:int}

Consider the multi-term composite monotone inclusion problem 
\begin{align}
  \label{eq:probi}
    0\in \sum_{i=1}^n\,G_i^*(\underbrace{A_i+B_i+C_i + D_i}_{T_i}) G_i (z),
\end{align}
where, for each $i=1,\dots, n$, $G_i$ is bounded linear, $A_i$ is (set-valued) maximal monotone, $B_i$ is (point-to-point) 
monotone and Lipschitz, $C_i$ is (point-to-point) cocoercive and $D_i$ is (point-to-point) monotone and continuously
differentiable (more details later in Section \ref{sec:alg}). 
Problem \eqref{eq:probi} appears in different fields of applied mathematics and optimization, including inverse problems and machine learning (see, e.g.,~\cite{combettes2014Coupled,johnstone-stochastics.preprint21,mon.sva-newton.siam12}).

In this paper, we present and study a new version of the 
projective splitting (PS) algorithm (see, e.g.,~\cite{combettes2014Coupled,eck.sva-fam.mp08,eck.sva-gen.sjco09, eckstein2022forward}) for solving \eqref{eq:probi}. 
Our main algorithm (Algorithm \ref{alg_main} below) splits \eqref{eq:probi} into its $n$ pieces, by evaluating each linear operator $G_i$, as well as the Lipschitz and cocoercive operators $B_i$ and $C_i$, respectively, and by computing the resolvent of $A_i$ plus a linearization of the continuously differentiable operator $D_i$, which we refer to as a proximal-Newton step.
Since PS algorithms are of high interest nowadays, it follows that special instances of \eqref{eq:probi} have been previously considered by other authors (more details later in this section), but up to our knowledge none of them considered proximal-Newton steps with respect to the smooth differentiable mappings $D_i$ in \eqref{eq:probi}.

\mgap

\noindent
{\bf The extended-solution set.} PS algorithms can be derived by noting that (under the additional assumption that $G_n=I$), 
a vector $z$ is a solution of problem \eqref{eq:probi} if and only if there exist $w_1,\dots, w_{n-1}$ such that  the $n$-tuple
$(z,w_1,\dots, w_{n-1})$ belongs to 
the \emph{extended-solution set}~\cite{combettes2014Coupled, eck.sva-gen.sjco09}:
\begin{align}
  \label{eq:def.Si}
    \mathcal{S}:=\left\{(z,w_1,\dots, w_{n-1}) \mid w_i\in 
T_i(G_iz),\; i=1,\dots,n,\; 
     w_n:=-\sum_{i=1}^{n-1}\,G_i^*w_i\right\}.
 \end{align}
Since $\mathcal{S}$ is a closed and convex subset of a product space (more details in Section \ref{sec:alg}), 
one can approximate points in $\mathcal{S}$ by successive projections onto separating hyperplanes. The basic idea consists in giving a current point $p^k=(z^k,w_1^k,\dots, w_{n-1}^k)$, find an affine functional/separator, say $\varphi_k(\cdot)$, such that
$\varphi_k(p^k)>0$ and $\varphi_k(p)\leq 0$ for all $p=(z,w_1,\dots, w_{n-1})\in \mathcal{S}$, and then define the next iterate
$p^{k+1}$ as the projection of $p^k$ (which can be explicitly computed) onto the semi-space $\{p\;|\;\varphi_k(p)\leq 0\}$ containing $\mathcal{S}$. 
For example, at a current iterate $k\geq 0$, a separator $\varphi_k(\cdot)$ can be constructed by choosing in an appropriate way $(x_i^k,y_i^k)$ in the graph of $T_i$ ($1\leq i\leq n$) and letting 
\begin{align}\lab{eq:varphi_i}
\varphi_k(p) := \sumn\,\inner{G_iz - x_i^k}{y_i^k - w_i},
\end{align}
where $p = (z,w_1,\dots, w_{n-1})$ and $w_n$ is as in \eqref{eq:def.Si}.

\mgap

\noindent
{\bf Sufficient conditions for separation.} One may ask under what conditions the affine functional $\varphi_k(\cdot)$ as in \eqref{eq:varphi_i} produces good separators, i.e., under what conditions we have 
$\varphi_k(p^k)>0$ and $\varphi_k(p)\leq 0$ for all $p\in \mathcal{S}$, while ensuring the resulting projective algorithm has good convergence properties. 
Since the definition of $\varphi_k(\cdot)$ depends on the points $(x_i^k,y_i^k)$ ($1\leq i\leq n$), 
all we need to do is to provide rules for picking these points in an appropriate way.
Some of the strategies from the current literature on the subject are as follows:
\begin{enumerate}
\item Assume that the resolvent $(T_i+I)^{-1}$ of each $T_i$ ($1\leq i\leq n$) is computable and let, for $i=1,\dots, n$,
\[
x_i^k = (\rho_i^k T_i + I)^{-1}(G_i z^k+\rho_i^k w_i^k)\;\;\mbox{and}\;\; 
y_i^k = \dfrac{G_i z^k - x_i^k}{\rho_i^k} + w_i^k, 
\]
where $\rho_i^k>0$ is a proximal parameter. This appears, for instance, 
in the works~\cite{combettes2014Coupled, eck.sva-fam.mp08,eck.sva-gen.sjco09}.
\item Set, for all $i=1,\dots, n$,  $B_i$ and $D_i$ as zero in \eqref{eq:probi}, and then compute single-forward steps for the cocoercive operators $C_i$ ($1\leq i\leq n$) and backward steps for the maximal monotone operators $A_i$, i.e.,
\[
x_i^k = (\rho_i^k A_i + I)^{-1}\left(G_i z^k+\rho_i^k w_i^k - \rho_i^kC_i(G_i z^k)\right)
\;\;\mbox{and}\;\; 
y_i^k = \dfrac{G_i z^k - x_i^k}{\rho_i^k} + w_i^k.
\]
This appears in the works~\cite{combettes2022saddle, joh.eck-single.coap21}.
\item Set, for all $i=1,\dots, n$,  $C_i$ and $D_i$ as zero in \eqref{eq:probi}, and then compute forward steps for the Lipschitz operators $B_i$ ($1\leq i\leq n$), in the spirit of the seminal work of P. Tseng~\cite{tse-mod.sjco00}, and backward steps for $A_i$, i.e., 
\[
x_i^k = (\rho_i^k A_i + I)^{-1}\left(G_i z^k+\rho_i^k w_i^k - \rho_i^k B_i(G_i z^k)\right)
\;\mbox{and}\;\; 
y_i^k = \dfrac{G_iz^k-x_i^k}{\rho_i^k} + w_i^k + \left[B_i(x_i^k)-B_i(G_iz^k)\right].
\]
Note that, when compared to $y_i^k$ as in item 2 above, $y_i^k$ as in the latter displayed equation  requires two more evaluations of the Lipschitz operator $B_i$.
This approach appeared, e.g., in the works~\cite{combettes2022saddle, eckstein2022forward}.

\end{enumerate}

\mgap

\noindent
{\bf Main contributions of this work.} As was already mentioned, the main contribution of this work is the incorporation of proximal-Newton steps in the PS framework with respect to the smooth differentiable mappings 
$D_i$ ($1\leq i\leq n$) in \eqref{eq:probi}. 
More precisely, to define the separator $\varphi_k(\cdot)$, 
we will compute pairs $(x_i^k, y_i^k)$ ($1\leq i\leq n$) as described in lines \ref{ln:line515} 
and \ref{ln:line514}-\ref{ln:line513} of Algorithm \ref{alg_main}, where, in particular, we use a linearization 
$D_{i \,(G_i z^k)}$ of $D_i$ at a (current) point $G_i z^k$ ($1\leq i\leq n$); this is exactly what we refer to as a proximal-Newton step.
We also mention that a line-search procedure (see line \ref{ln:line515} of Algorithm \ref{alg_main}) may be needed to choose the proximal parameter $\rho_i^k>0$ ($1\leq i\leq n$), as described in 
Algorithm \ref{alg_proc}.
The weak convergence of our main algorithm, namely Algorithm \ref{alg_main}, will be presented in Theorem \ref{th:main}.
For more details, we also refer the reader to the various comments/remarks following Algorithm \ref{alg_main}.

Next we review some related works on projective- and operator-splitting algorithms. 
As we already mentioned, none of the variants of PS algorithm mentioned below considered differentiable operators $D_i$ in 
$T_i$ - as in \eqref{eq:probi} - and, consequently, none of them performed proximal-Newton steps as in the present work. 

\mgap

\noindent
{\bf Projective- and operator-splitting algorithms.} To the best of our knowledge, the basic idea of PS algorithms originates in~\cite{eck.sva-fam.mp08}, with $n=2$ in \eqref{eq:probi}, $G_i=I$ and $T_i$ general maximal monotone operators ($i=1,2$), i.e., for solving the two-operator monotone inclusion $0\in T_1(z)+T_2(z)$.  Subsequently, the method 
was extended in~\cite{eck.sva-gen.sjco09} for $n>2$, but still under the above mentioned assumptions on $G_i$ 
and $T_i $ ($i=1,\dots, n$).
Compositions with linear operators $G_i$ ($i=1,\dots, n$) as in \eqref{eq:probi} appeared for the first time in~\cite{combettes2014Coupled} (see also
~\cite{combettes2015BestAppr,combettes2020warped,combettes2022saddle,combettes2016Solving}).
Asynchronous and block-iterative versions of the PS framework were developed and studied for the first 
time in~\cite{combettes2018asynchronous}.
Forward and single-forward steps for Lipschitz and cocoercive operators, respectively, were introduced and 
studied in~\cite{joh.eck-single.coap21, eckstein2022forward} (see also~\cite{eckstein2019rates}).
Stochastic and inertial versions of the PS algorithm 
were also studied in~\cite{johnstone-stochastics.preprint21} and ~\cite{alv.ger.mar-ps}, respectively.

The literature on operator-splitting algorithms for monotone inclusions is huge. The modern treatise~\cite{bau.com-book} is quite helpful in presenting, among other results on convex analysis and monotone operator theory, the basic theory of some of the most important operator-splitting algorithms like the celebrated 
Douglas-Rachford~\cite{eck.ber-dou.mp92, lio.mer-spl.sjna79} and Forward-Backward~\cite{lio.mer-spl.sjna79, pas-erg.jmaa79} splitting algorithms.
In the last decades, with a growing interest in applications in inverse problems and machine learning, many other variants of operator splitting algorithms have emerged. We mention the Davis–Yin splitting~\cite{davis2017three}, new variants of the alternating direction method of multipliers (ADMM) 
algorithm (see, e.g., \cite{boy.par.chu-dis.ftml11, he.convergence2012, mon.sva-blo.siam13}) and the hybrid proximal-extragradient (HPE) method and its many 
variants (see, e.g., \cite{alv-nesterov, alv.mar-ine.svva2020, att.alv.sva-dyn.jca16, bach-preprint21, jordan-controlpreprint20, mon.sva-hpe.siam10, mon.sva-newton.siam12} and references there in). 
Many other researchers and research groups have contributed to the subject in different directions, as for instance~\cite{
att-convergence.mp2020,
att-inertial.jota2018,
attouch.peypouquet-convergence.mp2019,
bauschke2020behaviour,
bec.teb-fas.sjis09,
bot.acm2019,condat2013pd,
giselson2021nonlinear,
glowinski2016splitting,
sun.equivalence2021,
hager.coap2020,
mahey2017survey,
lorenz2015inertial,
lorenz.nonstationay2019,
salzo2012inexact,
salzo2022parallel}.
The monographs~\cite{boy.par.chu-dis.ftml11, ryu.win-book} can also be quite useful to many of the
above mentioned algorithms and results. 

\mgap

\noindent
{\bf Organization of the paper.} 
The material is organized as follows.
In Section \ref{sec:pre}, we present some basic results. 
In Section \ref{sec:alg}, we present our main algorithm, namely Algorithm \ref{alg_main}. 
In Section \ref{sec:conv}, Theorem \ref{th:main}, we study the weak convergence of Algorithm \ref{alg_main}.  
Section \ref{sec:bbp} is devoted to present and study the bracketing/bisection procedure called during the execution of Algorithm \ref{alg_main}; see Algorithm \ref{alg_proc}.

\mgap

\noindent
{\bf General notation.}
Let $\HH$ be a real Hilbert space.
The weak convergence of a sequence $(z^k)$ in $\HH$ to $z\in \HH$ will be denoted by $z^k\wto z$.
A set-valued map $T:\HH\tos \HH$ is said to be a \emph{monotone operator} if
$\inner{z-z'}{v-v'}\geq 0$ for all $v\in T(z)$ and $v'\in T(z')$. On the other hand, $T:\HH\tos \HH$ is
\emph{maximal monotone} if $T$ is monotone and its \emph{graph}
$\mbox{Gr}(T):=\{(z,v)\in \HH\times \HH\,|\,v\in T(z)\}$ is not properly contained in the graph of any other
monotone operator on $\HH$. 
The \emph{domain} of $T:\HH\tos \HH$ is $\mbox{Dom}(T):=\{z\in \HH\;|\;T(z)\neq \emptyset\}$.
The \emph{inverse} of $T:\HH\tos \HH$ is $T^{-1}:\HH\tos \HH$, defined at any
$z\in \HH$ by $v\in T^{-1}(z)$ if and only if $z\in T(v)$. 
The \emph{resolvent} of a maximal monotone operator
$T:\HH\tos \HH$ is $J_T:= (T+I)^{-1}$; note that $z=(T+I)^{-1}x$ if and only if $x-z\in T(z)$. 
The operator $\gamma T:\HH\tos \HH$, where $\gamma>0$, is defined by $(\gamma T)z:=\gamma T(z):=\{\gamma v\,|\,v\in T(z)\}$.
The sum of $T:\HH\tos \HH$ and $S:\HH\tos \HH$ is defined via the usual Minkowski sum 
$T(z)+S(z):=\{u+v\;|\;u\in T(z)\;\;\mbox{and}\;\;v\in S(z)\}$.

The \emph{adjoint} of a bounded linear operator $G:\HH\to \HH$ is denoted by $G^*$. A (point-to-point) operator $C:\HH\to \HH$ is said to be $\beta$-\emph{coercive} if $\beta\geq 0$ and $\inner{z-z'}{C(z)-C(z')}\geq \beta\normq{C(z)-C(z')}$ for all
$z,z'\in \HH$.
On the other hand, a (point-to-point) operator $B:\HH\to \HH$ is said to be $\ell$-Lipschitz continuous if $\ell\geq 0$
and $\norm{B(z)-B(z')}\leq \ell\norm{z-z'}$ for all $z,z'\in \HH$.
The \emph{subdifferential} of a convex function $f:\HH\to (-\infty, \infty]$ is $\partial f:\HH\tos \HH$, defined at
any $z\in \HH$ as $\partial f(z):=\{u\in \HH\;|\;f(z')\geq f(z)+\inner{u}{z'-z}\;\;\;\forall z'\in \HH\}$. For additional details on 
standard notation on convex analysis and monotone operators we refer the reader to the 
books~\cite{bau.com-book,rock-ca.book}.

The derivative of a continuously differentiable mapping $D:\HH\to \HH$ will be denoted by $D'(\cdot)$. The norm of 
$D'(z)$ is $\norm{D'(z)}:=\sup_{\norm{x}\leq 1}\,\norm{D'(z)x}$. We will also frequently use the notation
\begin{align} \label{eq:lin.app}
 D_{(u)}(z):=D(u)+D'(u)(z-u)\qquad \forall z,u\in \HH.
\end{align}

Now let $\HH_0, \HH_1,\ldots, \HH_{n-1}$ be real Hilbert spaces, let $\HH_n:=\HH_0$ and let $\inner{\cdot}{\cdot}$ and
 $\|\cdot\|=\sqrt{\inner{\cdot}{\cdot}}$ denote the inner product and norm (respectively) in $\HH_i$ ($i=0,\dots, n$).
 Let $\HHH:=\HH_0\times \HH_1\times \dots\times \HH_{n-1}$ be endowed with the inner product and norm
defined, respectively, as follows (for some $\gamma >0$):
\begin{align}
  \label{eq:def.inner}
\inner{p}{\tilde p}_{\gamma}=\gamma\inner{z}{\tilde z}+\sum_{i=1}^{n-1}\inner{w_i}{\tilde w_i},\quad
\norm{p}_\gamma = \sqrt{\inner{p}{p}_\gamma},
\end{align}
where $p=(z, w)$, $\tilde p = (\tilde z, \tilde w)$, with $z,\tilde z\in \HH_0$ and $w:=(w_1,\dots, w_{n-1}), \tilde w:=(\tilde w_1,\dots, \tilde w_{n-1})\in \HH_1\times \ldots\times \HH_{n-1}$.
Elements in $\HHH$ will be denoted either by
$p=(z, w_1,\dots, w_{n-1})$ or simply $p=(z,w)$, where $w=(w_1,\dots, w_{n-1})$, and we will use the notation
 \begin{align}
 \label{eq:def.wn}
 w_n = -\sum_{i=1}^{n-1}\,G_i^* w_i
\end{align}
whenever $p\in \HHH$ and $G_i$ is bounded linear ($i=1,\dots, n-1$).

We will also adopt the following conventions: $0\cdot\infty=0$, $1/ \infty=0$, $1/0=\infty$
 and $r\cdot \infty=\infty$ for all $r\in (0,\infty)$.
By $\R_{++}$ we denote the set of strictly positive real numbers.

\section{Basic results}
 \label{sec:pre}

In this section, we present a general separator-projection framework (Algorithm \ref{alg_gen}) for 
finding a point in a nonempty closed and convex subset of a Hilbert space, along with other basic results we will need in this paper. 
The main reason to consider Algorithm \ref{alg_gen} here comes from the 
fact (as previously discussed in Section \ref{sec:int}) that the monotone inclusion 
problem \eqref{eq:probi} can be reformulated as the problem of finding a point in the extended solution set $\mathcal{S}$ as in \eqref{eq:def.Si}.  
Algorithm \ref{alg_gen} will be used in Section \ref{sec:conv} to analyze the convergence of the main algorithm proposed in 
this paper, namely Algorithm \ref{alg_main} below.

Throughout this section, $\HH$ and $\mathcal{K}$ denote real Hilbert spaces with inner product $\inner{\cdot}{\cdot}$
and norm $\|\cdot\|=\sqrt{\inner{\cdot}{\cdot}}$. 
We denote the gradient of an affine function $\varphi:\HH\to \R$ by the standard notation $\nabla \varphi$ 
and we also write $\varphi(z)=\inner{\nabla \varphi}{z}+\varphi(0)$ for all $z\in \HH$.

\mgap

\begin{algorithm}[H]
 \caption{Generic linear separation-projection method for finding a point in a nonempty closed and convex set 
$\mathcal{S}\subset \HH$}
\label{alg_gen}
\SetAlgoLined
\KwInput{$p^0\in \HH$, $0<\underline{\tau}<\overline{\tau}<2$ }
 \For{$k=0,1,2,\dots$}{
  Find an affine function $\varphi_k$ in $\HH$ such that 
  $\mathcal{S}\subset \{p\in \HH\;|\;\varphi_k(p)\leq 0\}$\\
  \eIf{$\varphi_k(p^k)>0$}
      {Choose $\tau_k\in [\underline{\tau},\overline{\tau}]$\\[2mm]
      $\alpha_k = \tau_k\dfrac{\varphi_k(p^{k})}{\norm{\nabla \varphi_k}^2}$\\[3mm]
     $p^{k+1}= p^{k} - \alpha_k\nabla \varphi_k$
     }
     {$p^{k+1}= p^{k}$}
  }
\end{algorithm}

\mgap

Next lemma is about the weak convergence of Algorithm \ref{alg_gen}.

\begin{lemma}\emph{(\cite{combettes2014Coupled})}
 \label{lm:alg_gen}
  The following holds for any sequence $(p^k)$ generated by \emph{Algorithm \ref{alg_gen}}.
   \begin{itemize}
    \item[\emph{(a)}] $(p^k)$ is bounded.
    \item[\emph{(b)}] If all weak cluster points of $(p^k)$ are in $\mathcal{S}$, then $(p^k)$ converges weakly to some
    point in $\mathcal{S}$.
  \item[\emph{(c)}] If there exists $\xi>0$ such that $\norm{\nabla \varphi_k}\leq \xi$, for all $k\geq 0$,  then 
 $\overline{\lim} \,\varphi_k(p^k)\leq 0$.
\end{itemize}
\end{lemma}

\mgap

Next lemma will be useful in the convergence analysis of Algorithm \ref{alg_main}.

\begin{lemma}\emph{(\cite[Lemma 3.1]{com-pri.ol14})} \lab{lm:pro_nuc}
Let $G:\HH\to \mathcal{K}$ be a bounded linear operator let $V$ be the graph of $G$, i.e., 
$V:=\{(z,Gz)\in \HH\times \mathcal{K}\;|\; z\in \HH\}$.
Let $P$ and $P^\perp$ denote the orthogonal projections onto $V$ and $V^\perp$ (the orthogonal complement of $V$), respectively. Then, for all $(z,w)\in \HH\times \mathcal{K}$, we have
\begin{align*}
& P(z, w) = ((I+G^*G)^{-1}(z+G^*w), G(I+G^*G)^{-1}(z+G^*w)),\\
& P^\perp(z,w) = ( G^*(I+GG^*)^{-1}(Gz-w), -(I+GG^*)^{-1}(Gz-w)).
\end{align*}
\end{lemma}

\mgap

A proof of the lemma below in finite-dimensional spaces can be found in \cite[Lemma 4.1.12]{den.sch-numbook.siam96}. 
A proof in general Hilbert spaces follows the same outline as in the finite-dimensional case.

\mgap

\begin{lemma}
 \label{lm:lip}
 Let $D:\HH\to \HH$ be a continuously differentiable mapping
  with an $m$-Lipschitz continuous derivative $D'(\cdot)$, i.e., let $D$ be such that there exits $m\geq 0$ satisfying
\begin{align}
 \label{eq:lip.app}
 \norm{D'(x)-D'(y)}\leq m\norm{x-y}\qquad  \forall x,y \in \HH.
\end{align}
Then,
  \[
   \norm{D(z) - D_{(u)}(z)}\leq \dfrac{m}{2}\norm{z-u}^2 \qquad \forall z,u\in \HH,
  \]
  where $D_{(u)}(z)$ is as in \eqref{eq:lin.app}.
\end{lemma}

\mgap

\begin{lemma}
\label{lm:bound}
Let $T:\HH\tos \HH$ be a maximal monotone operator.
Let $(s^k)$ and $(u^k)$ be bounded sequences in $\HH$, let $(\rho^k)$ be a bounded sequence
in $\R_{++}$ and let $D_{(u^k)}(\cdot)$ be as in \eqref{eq:lin.app} with $u=u^k$. 
Then the sequence $(x^k)$ defined by
\begin{align}
  \label{eq:galeao} 
 x^k=J_{\rho^k\left(T+D_{(u^k)}\right)}(s^k)\qquad \forall k\geq 0
\end{align}
is bounded.
\end{lemma}
\begin{proof}
  Let $w\in \mbox{Dom}(T)$ and, for each $k\geq 0$, choose $r^k\in \HH$ such that $r^k\in (T+D_{(u^k)})(w)$, i.e.,
  take $a\in T(w)$ and let, for each $k\geq 0$, $r^k=a+D_{(u^k)}(w)$. 
Since $(u^k)$ is bounded, it follows from Lemma \ref{lm:lip} that $(r^k)$ is also bounded.
Note now that the inclusion $r^k\in (T+D_{(u^k)})(w)$ is clearly equivalent
to $w=J_{\rho^k\left(T+D_{(u^k)}\right)}(\rho^k r^k + w)$, which combined with \eqref{eq:galeao} and the nonexpansiveness
of $J_{\rho^k (T+D_{u^k})}(\cdot)$ yields
\[
 \norm{x^k-w}\leq \norm{s^k - \rho^k r^k -w}\leq \norm{s^k-w}+\rho^k\norm{r^k}.
\]
The desired result now follows from the boundedness of $(s^k)$, $(\rho^k)$ and $(r^k)$.
\end{proof}

\section{The problem and the main algorithm}
 \label{sec:alg}
 
 In this section, we present our main algorithm, namely Algorithm \ref{alg_main}, 
 for solving the multi-term monotone inclusion problem \eqref{eq:probi}. 
 In some cases, Algorithm \ref{alg_main} may need a bracketing/bisection procedure, which we discuss in Section \ref{sec:bbp} (see Algorithm \ref{alg_proc}). 
 The main result on the weak convergence of Algorithm \ref{alg_main} will be presented in 
 Section \ref{sec:conv} (Theorem \ref{th:main}). We also refer the reader to the comments and remarks following Algorithms \ref{alg_main} and \ref{alg_proc}.
 
Consider the monotone inclusion problem \eqref{eq:probi}, i.e., the problem of finding
$z\in \HH_0$ such that
\begin{align}
  \label{eq:prob}
    0\in \sum_{i=1}^n\,G_i^*(\underbrace{A_i+B_i+C_i + D_i}_{T_i}) G_i (z),
\end{align}
where the following assumptions are assumed to hold for all $i=1,\dots, n$:
\begin{itemize}
\item [(A1)] $G_i:\HH_0\to \HH_i$ is a bounded linear operator and $G_n= I$ (the identity operator).
\item [(A2)] $A_i:\HH_i\tos \HH_i$ is a (set-valued) maximal monotone operator.
\item [(A3)] $B_i:\HH_i\to \HH_i$ is a (point-to-point) monotone and $\ell_i$-Lipschitz continuous operator, i.e., there
 exists $\ell_i\in [0,\infty)$ such that
     \begin{align}
        \label{eq:prop.coco}
          \norm{B_i(x)-B_i(y)}\leq \ell_i \norm{x-y}\qquad \forall x,y\in \HH_i.
     \end{align}
\item [(A4)] $C_i:\HH_i\to \HH_i$ is a (point-to-point) $\beta_i$-cocoercive operator, i.e., there exists 
$\beta_i\in (0,\infty]$ such that
     \begin{align}
        \label{eq:prop.cocox}
          \inner{x-y}{C_i(x)-C_i(y)}\geq \beta_i\norm{C_i(x)-C_i(y)}^2\qquad \forall x,y\in \HH_i.
     \end{align}
\item [(A5)] $D_i:\HH_i\to \HH_i$ is a (point-to-point) continuously differentiable monotone operator with an
$m_i$-Lipschitz continuous derivative $D'_i(\cdot)$, i.e., there exists $m_i\in [0,\infty)$ such that
     \begin{align}
        \label{eq:prop.smooth}
          \norm{D'_i(x)-D'_i(y)}\leq m_i \norm{x-y}\qquad \forall x,y\in \HH_i.
     \end{align}
\item [(A6)] The solution set of \eqref{eq:prob} is nonempty.
\end{itemize}

We also assume the following regarding the family of operators $D_i(\cdot)$ ($i=1,\dots, n$):
\begin{itemize}
\item [(A7)] 
There exists a subset of indexes $\mathcal{I}_{\mathcal{D}}\subset \{1,\dots, n\}$ such that
$m_i>0$ whenever $i\in \mathcal{I}_{\mathcal{D}}$, and 
$m_i=0$ and $D_i=0$ for all $i\in \{1,\dots, n\}\backslash \mathcal{I}_{\mathcal{D}}$ ($m_i=0$ in \eqref{eq:prop.smooth} implies that $D_i$ is affine and, as a consequence, it can be incorporated to either $B_i$ or $C_i$ and so ignored in
\eqref{eq:prob}, i.e., set to be zero).
\end{itemize}

An important instance (under appropriate qualification conditions) of \eqref{eq:prob} 
is the minimization problem
\begin{align}
 \label{eq:prob_min}
 \min_{z\in \HH_0}\,\sum_{i=1}^n\,\Big\{f_i(G_i z) + g_i(G_iz) + h_i(G_iz)\Big\},
\end{align}
where, for $i=1,\dots, n$, $f_i:\HH_i\to (-\infty,\infty]$ is proper, convex and lower semicontinuous, 
$g_i:\HH_i\to \R$ is convex differentiable with a 
$\beta^{-1}_i$-Lipschitz continuous gradient $\nabla g_i$ and 
$h_i:\HH\to \R$ is convex and twice-continuously differentiable with an $m_i$-Lipschitz continuous second derivative $\nabla^2 h_i$,
in which case we have, in the general setting of \eqref{eq:prob}:
\begin{align}
  \label{eq:op_fun}
  A_i = \partial f_i,\quad B_i= 0,\quad C_i=\nabla g_i\;\;\mbox{and}\;\; D_i=\nabla h_i.
\end{align}

We mention that several problems of interest in Machine Learning, including \emph{logistic regression problems}, are special instances of \eqref{eq:prob_min} -- see, e.g., \cite{Bach2010, FereyBachRudi2019}.

Our main Algorithm (see Algorithm \ref{alg_main} below) for solving \eqref{eq:prob} will, 
in particular, be conformed to the following set of rules (for $i=1,\dots n$):
\begin{enumerate}
\item Evaluation of each linear operator $G_i$.
\item Explicit evaluation of the Lipschitz and cocoercive components $B_i$ and $C_i$, respectively.
\item Computation and evaluation of the first-order approximation of $D_i$:
\begin{align}
  \label{eq:def.fo}
  D_{i\;(u)}(z):=D_i(u)+D_i'(u)(z-u)\qquad (z,u\in \HH_i).
\end{align}
\item Implicit computation of the resolvent $J_{\rho(A_i+D_{i\;(u)})}:=\left(\rho(A_i+D_{i\;(u)})+I\right)^{-1}$ of the operator $A_i+D_{i\;(u)}$, where $\rho>0$ 
(see the third remark following Algorithm \ref{alg_main} for additional details).
\end{enumerate} 

Recall that we assumed the solution set of \eqref{eq:prob} to be nonempty, i.e., we assumed that there exists at least one $z\in \HH_0$ satisfying the inclusion in \eqref{eq:prob}. 
We also note that, under assumptions (A2)--(A5) above, $T_i$ as defined in \eqref{eq:prob} is maximal monotone
for all $i=1,\dots, n$.  

Before presenting our algorithm, recall the \emph{extended-solution set} as previously presented in \eqref{eq:def.Si}:
\begin{align}
  \label{eq:def.S}
    \mathcal{S}:=\left\{p\in \HHH \mid w_i\in 
(\underbrace{A_i+B_i+C_i+D_i}_{T_i})(G_iz),\; i=1,\dots,n,\; 
     w_n:=-\sum_{i=1}^{n-1}\,G_i^*w_i\right\},
 \end{align}
where by $p$ we denote $(z,w_1,\dots, w_{n-1})$. One can prove that $\mathcal{S}$ is closed and convex 
(see, e.g.,~\cite{combettes2014Coupled}).

\mgap
\mgap

Next is our main algorithm.

%
\begin{algorithm}[H]
 \caption{Projective splitting with backward, half-forward and proximal-Newton steps for solving problem \eqref{eq:prob}}
\label{alg_main}
\SetAlgoLined
\KwInput{$(z^0, w_1^0,\dots, w_{n-1}^0)\in \HHH$, $0<\underline{\tau}<\overline{\tau}<2$, 
$0<\underline{\theta}<\overline{\theta}<2$, $\hat \rho>0$, $\hat \delta>0$, $\gamma>0$ and $w_n^0:=-\sum_{i=1}^{n-1}\,G_i^*w_i^0$}
 \For{$k=0,1,2,\dots$}{
 \For{$i=1,\dots, n$}{
   \eIf{$w_i^k\in (A_i+B_i+C_i+D_i)(G_i z^k)$}
        {$\rho_i^k=\hat \rho$, $x_i^k=G_i z^k$ and $y_i^k = w_i^k$\\ \label{ln:line600}}
        {
        \eIf
        {$i\in \mathcal{I}_\mathcal{D}$}
        {Let $(\rho_i^k, x_i^k, y_i^k)\gets 
            \FMain{$A_i,B_i,C_i,D_i,G_i,z^k,w_i^k,\ell_i,\beta_i, m_i$}$, i.e.,\\[3mm]\label{ln:line515}
            $x_i^k = J_{\rho_i^k \left(A_i+D_{i \,(G_i z^k)}\right)} \left(G_iz^k + \rho_i^k w_i^k - \rho_i^k  
            (B_i+C_i)(G_i z^k)\right)$\\[3mm]\label{ln:line508}
           $y_i^k = \dfrac{G_iz^k-x_i^k}{\rho_i^k} + w_i^k + \left[B_i(x_i^k)-B_i(G_iz^k)\right] 
              + \left[D_i(x_i^k) - D_{i \,(G_i z^k)}(x_i^k)\right]$\\[3mm] \label{ln:line509}
            $\underline{\theta}\leq 4\,\ell_i^2 (\rho_i^k)^2 + \left(\beta_i^{-1}+\hat \delta\right)\rho_i^k +
            \left(m_i \rho_i^k \norm{x_i^k - G_i z^k}\right)^2\leq \overline{\theta}$\\ \label{ln:line510}
           }
          {Choose $\rho_i^k>0$\\[3mm]\label{ln:line514}
            $x_i^k = J_{\rho_i^k A_i} \left(G_iz^k + \rho_i^k w_i^k - \rho_i^k  
            (B_i+C_i)(G_i z^k)\right)$\\[3mm]\label{ln:line512}
           $y_i^k = \dfrac{G_iz^k-x_i^k}{\rho_i^k} + w_i^k + \left[B_i(x_i^k)-B_i(G_iz^k)\right]$\\ \label{ln:line513}
        }
        }
        $u_i^k = x_i^k - G_i x_n^k$\\ \label{ln:line501}
        }
$v^k  = \sum_{i=1}^{n-1}\,G_i^*y_i^k + y_n^k$ \\[3mm] \label{ln:line502}  
$\varphi_k = \inner{z^k}{v^k} + \sum_{i=1}^{n-1}\,\inner{w_i^k}{u_i^k}
             -\sum_{i=1}^n\,\left[\inner{x_i^k}{y_i^k}
            +\dfrac{1}{4\beta_i}\norm{x_i^k - G_i z^k}^2\right]$\\[3mm] \label{ln:line700}
$\pi_k = \gamma^{-1}\norm{v^k}^2 + \sum_{i=1}^{n-1}\norm{u_i^k}^2 $\\[3mm] \label{ln:line503}
   \eIf{$\varphi_k>0$}{
           Choose $\tau_k\in [\underline{\tau},\overline{\tau}]$\\[3mm]
            $\alpha_k  = \tau_k\dfrac{\varphi_k}{\pi_k}$\\[3mm] \label{ln:line506}
            $z^{k+1} = z^k - \gamma^{-1}\alpha_k v^k$\\[3mm] \label{ln:line504}
           \For{$i=1,\dots, n-1$}
           {$w_i^{k+1} = w_i^k -\alpha_k u_i^k$\\ \label{ln:line505}}
           }{
            $z^{k+1} = z^k$\\[3mm]  \label{ln:line500}
            \For{$i=1,\dots, n-1$}
           {$w_i^{k+1} = w_i^k$\\ \label{ln:line505b}}
           }
$ w_n^{k+1} = -\sum_{i=1}^{n-1}\,G_i^* w_i^{k+1}$ \\ \label{ln:line511}
 }
\end{algorithm}
%

\mgap

\noindent
Next we make some remarks regarding Algorithm \ref{alg_main}:

\begin{itemize}
\item[(i)] If one is lucky enough to find $(G_i z^k, w_i^k)$ in the graph of $T_i$, as in line \ref{ln:line600}, 
then one simply set $x_i^k$ and $y_i^k$ as $G_i z^k$ and $w_i^k$, respectively,  and keep $\rho_i^k$ equal to the constant 
$\hat \rho>0$.
\item[(ii)] Note that if $i\in \mathcal{I}_{\mathcal{D}}$ we may need a bracketing/bisection procedure to find
a triple $(\rho_i^k, x_i^k,y_i^k)$ satisfying the conditions in lines \eqref{ln:line508}-\eqref{ln:line510}. As we will prove in Section \ref{sec:bbp}, the bracketing/bisection procedure (Algorithm \ref{alg_proc}) terminates after a finite number of steps - see Proposition \ref{pr:bterminates}. On the other hand, if $i\notin \mathcal{I}_{\mathcal{D}}$, in which case 
$D_i=0$, no bracketing/bisection procedure is required, and the computation of $(\rho_i^k, x_i^k,y_i^k)$ will depend only on the evaluation of $J_{A_i}$, $G_i$, $B_i$ and $C_i$ (see lines \ref{ln:line514}-\ref{ln:line513}).
We also mention that the problem of removing this bracketing/bisection (in the context of monotone 
inclusions for a single operator) has been posed as a challenging and difficult open problem
(see, e.g., \cite{bullins2022lai, jiang2022mokhtari, jordan-controlpreprint20, jordan-perseus22, 
Lin2023Jordan, Nesterov2018Book}).
\item[(iii)] Now we discuss the computation of resolvent $J_{\rho(A_i+D_{i\;(u)})}:=\left(\rho(A_i+D_{i\;(u)})+I\right)^{-1}$, as in line \ref{ln:line508}.
Note that computing $x^{+}=J_{\rho(A_i+D_{i\;(u)})}(y)$ is equivalent to solving the regularized inclusion problem
\begin{align*}
 0\in \rho(A_i(x)+D_{i\;(u)}(x)) + x -y, 
\end{align*}
which, in the setting of the optimization problem \eqref{eq:prob_min} (see also \eqref{eq:op_fun}), gives $x^{+}$ as the solution of the following (second-order in $h_i$) regularized model
\begin{align}
 \label{eq:prob_quad}
 \min_{x\in \HH_i}\,\left\{f_i(x)+h_i(u)+\inner{\nabla h_i(u)}{x-u}+
\dfrac{1}{2}\inner{\nabla^2 h_i(u)(x-u)}{x-u} + \dfrac{1}{2\rho}\norm{x-y}^2\right\}.
\end{align}

Hence, in the optimization setting \eqref{eq:prob_min}, the execution of Algorithm \ref{alg_main} will be, in particular, supported on the assumption that one knows how to solve \eqref{eq:prob_quad} efficiently, for any given 
$y,u \in \HH_i$ and $\rho>0$.

More precisely, still in the optimization setting, $x_i^k$ as in line \ref{ln:line508} is given as 
\begin{align}
 \label{eq:prob_quad02}
\nonumber
& x_i^k=\mbox{arg}\min_{x\in \HH_i}\,\Bigg\{f_i(x)+h_i(G_iz^k)+\inner{\nabla h_i(G_iz^k)}{x-G_iz^k}+\\
&\dfrac{1}{2}\inner{\nabla^2 h_i(G_iz^k)(x-G_iz^k)}{x-G_iz^k} + 
\dfrac{1}{2\rho_i^k}\left\|x-\left[G_iz^k + \rho_i^k w_i^k - \rho_i^k  
            (B_i+C_i)(G_i z^k)\right]\right\|^2\Bigg\},
\end{align}
which, if $\mathcal{I}_\mathcal{D}\neq \emptyset$, represents the computational burden of Algorithm \ref{alg_main}. In this case, $y_i^k$ in line \ref{ln:line509} is given explicitly as (recall that $B_i=0$; see \eqref{eq:op_fun}):
\begin{align*}
y_i^k = \dfrac{G_iz^k-x_i^k}{\rho_i^k} + w_i^k
              + \left[\nabla h_i(x_i^k)-\nabla h_i(G_iz^k)-\nabla^2 h_i(G_iz^k)(x_i^k-G_iz^k)\right].
\end{align*}
Note that problem \eqref{eq:prob_quad02} can be solved by first-order methods, for example, under the standard assumption that the proximal operator $(\lambda \partial f_i + I)^{-1}$ of each $f_i$ ($1\leq i\leq n$) can be efficiently computed (see, e.g., \cite{CombettesPesquet2011, Nesterov2018Book}).
\item[(iv)] The condition in line \ref{ln:line510}, namely
\[
 \underline{\theta}\leq 4\,\ell_i^2 (\rho_i^k)^2 + \left(\beta_i^{-1}+\hat \delta\right)\rho_i^k +
            \left(m_i \rho_i^k \norm{x_i^k - G_i z^k}\right)^2\leq \overline{\theta},
\]
will be important in the proof of Proposition \ref{pr:geq} below (see Eq. \eqref{eq:djairo02}). 
\item[(v)] If $\mathcal{I}_\mathcal{D} = \emptyset$, that is, if there is no differentiable components $D_i$ ($1\leq i\leq n$) in 
\eqref{eq:prob}, then it follows that Algorithm \ref{alg_main} is closely related to the main 
algorithms in~\cite{combettes2022saddle,joh.eck-single.coap21,eckstein2022forward}. On the other hand, as we have already discussed, none of these works have considered the framework given by problem \eqref{eq:prob}, under the 
Assumptions (A1)-(A7) above (see also Assumption (A8) below).
\end{itemize}

\section{Convergence analysis of Algorithm \ref{alg_main}}
 \label{sec:conv}

The main result in this section is Theorem \ref{th:main} on the weak convergence of Algorithm \ref{alg_main}.
The main result in Proposition \ref{pr:special} is given in item (e), which says that Algorithm \ref{alg_main} is a special instance of the separator-projection framework given in Algorithm \ref{alg_gen}.
Next, under the additional Assumption (A8) - see \eqref{eq:conway} below- on the proximal parameters $\rho_i^k$ ($1\leq i\leq n$) we prove, in Lemma \ref{lm:galeao02}, several results on the boundedness of some of the sequences evolved by Algorithm \ref{alg_main}.
Lemma \ref{lm:dantzig} is an auxiliary result and Proposition \ref{pr:geq} will be useful while proving our main theorem in this section, namely Theorem \ref{th:main}. Lemma \ref{lm:previous} yields a boundedness result on the gradients of the sequence of separators $\varphi_k(\cdot)$ - see \eqref{eq:def.phi2} - and some asymptotic results involving the 
sequences $(x_i^k)$, $(y_i^k)$ and $(w_i^k)$ ($1\leq i\leq n$).

We begin with the following technical lemma, which will be useful in the proofs of 
Proposition \ref{pr:special} and Theorem \ref{th:main}.

 \begin{lemma}
   \label{lm:fund}
     Let $(x_i^k)$, $(y_i^k)$ \emph{(}$i=1,\dots, n$\emph{)} and $(z^k)$ be generated by \emph{Algorithm \ref{alg_main}}.
    For any $(r_i, a_i)\in \emph{Gr}(A_i)$, $i=1,\dots, n$, we have
    \begin{align} 
       \label{eq:taylor}
     \inner{r_i - x_i^k}{y_i^k - (B_i+C_i+D_i)(r_i) - a_i}\leq \dfrac{1}{4\beta_i}
\norm{x_i^k - G_i z^k}^2\qquad (i=1,\dots, n).
    \end{align}
 \end{lemma}
 \begin{proof}
 First note that if $(x_i^k,y_i^k)$ is as in line \ref{ln:line600} of Algorithm \ref{alg_main}, then we obtain that the right-hand side of \eqref{eq:taylor} is equal to zero and that the inclusion $y_i^k \in (A_i+B_i+C_i+D_i)(x_i^k)$ hold. The latter combined with the inclusion $a_i\in A_i(r_i)$ and the monotonicity of
 $A_i+B_i+C_i+D_i$ yields the desired inequality in \eqref{eq:taylor} (since the right-hand side of \eqref{eq:taylor}, in this case, is zero).
 
Assume now that $(\rho_i^k,x_i^k,y_i^k)$ is as in lines \ref{ln:line508}-\ref{ln:line510}
(resp. \ref{ln:line514}-\ref{ln:line513})
of Algorithm \ref{alg_main}.
 From lines \ref{ln:line508} and \ref{ln:line509} (resp. \ref{ln:line512} and \ref{ln:line513}), the definition of 
$J_{\rho_i^k \left(A_i+D_{i \,(G_i z^k)}\right)}$ (resp. $J_{\rho_i^k A_i}$) and some simple algebra we find 
$y_i^k - (B_i+D_i)(x_i^k)-C_i(G_iz^k)\in A_i (x_i^k)$
(By Assumption (A7) in Section \ref{sec:alg}, we have $D_i=0$ whenever $i\notin \mathcal{I}_{\mathcal{D}}$.), i.e.,
\begin{align*}
(x_i^k,  y_i^k - (B_i+D_i)(x_i^k)-C_i(G_i z^k))\in \mbox{Gr}(A_i),
\end{align*}
which, in turn, when combined with the assumption that $(r_i,a_i)\in \mbox{Gr}(A_i)$ and the monotonicity of $A_i$ yields
\begin{align*}
 \inner{r_i - x_i^k}{y_i^k - (B_i+D_i)(x_i^k)-C_i(G_i z^k)  - a_i}\leq 0.
\end{align*}
The latter inequality combined with the monotonicity of $B_i + D_i$ gives
\begin{align}
\nonumber
\inner{r_i - x_i^k}{y_i^k  - (B_i+C_i+D_i)(r_i) - a_i} &= 
 \inner{r_i - x_i^k}{y_i^k - (B_i+D_i)(x_i^k) - C_i(G_i z^k) - a_i}\\
 \nonumber
  &\hspace{1cm}+\inner{r_i  - x_i^k}{(B_i+D_i)(x_i^k) - (B_i+D_i)(r_i)} \\
\nonumber
&\hspace{1cm}+\inner{r_i  - x_i^k}{C_i(G_iz^k)- C_i(r_i)}\\
  \label{eq:1137}
    & \leq \inner{r_i  - x_i^k}{C_i(G_iz^k)- C_i(r_i)}.
\end{align}
If $\beta_i< \infty$, using Assumption (A4), the Cauchy-Schwarz inequality and some algebraic manipulations,
we obtain
\begin{align*}
\inner{r_i  - x_i^k}{C_i(G_iz^k)- C_i(r_i)} &= \inner{r_i  - G_i z^k}{C_i(G_iz^k)- C_i(r_i)} + \inner{G_i z^k - x_i^k}{C_i(G_iz^k)- C_i(r_i)}\\
    &\leq -\beta_i\norm{C_i(G_iz^k)- C_i(r_i)}^2 + \norm{x_i^k - G_i z^k}\norm{C_i(G_iz^k)- C_i(r_i)},
\end{align*}
which combined with \eqref{eq:1137} and the fact that the maximum of the quadratic real function $t\mapsto -\beta_i t^2 + \norm{G_i z^k - x_i^k}\,t$ is $\norm{x_i^k - G_i z^k}^2/(4\beta_i)$ yields exactly \eqref{eq:taylor}.

On the other hand, if $\beta_i=\infty$, then using the conventions 
$r\cdot\infty=\infty$ if $r>0$ and $1/\infty=0$ (see ``General notation'' in Section \ref{sec:int}) and \eqref{eq:prop.cocox} with 
$x=G_i z^k$ and $y=r_i$, we obtain $C_i(G_iz^k)=C_i(r_i)$.
Hence, from \eqref{eq:1137},
\begin{align*}
\inner{r_i - x_i^k}{y_i^k - (B_i+C_i+D_i)(r_i) - a_i}\leq 0 = \dfrac{1}{4\beta_i}\norm{x_i^k - G_i z^k}^2,
\end{align*}
which finishes the proof of the lemma.
\end{proof}

\mgap
\mgap

Consider now the sequences evolved by 
 Algorithm \ref{alg_main} and define $\varphi_k:\HHH\to \R$ by
\begin{align}
\label{eq:def.phi2}
 & \varphi_k (\underbrace{z,w_1,\dots, w_{n-1}}_{p})
 =\inner{z}{v^k} + \sum_{i=1}^{n-1}\,\inner{w_i}{u_i^k}
  -\sum_{i=1}^n\,\left[\inner{x_i^k}{y_i^k}
   +\dfrac{1}{4\beta_i}\norm{x_i^k - G_i z^k}^2\right].
 \end{align}
Define also $(p^k)$ by
 \begin{align}
  \label{eq:def.pk}
 & p^k = (z^k,w_1^k,\dots, w_{n-1}^k),\qquad \forall k\geq 0.
 \end{align}

\mgap

Next we prove, in particular, that Algorithm \ref{alg_main} is a special instance of Algorithm \ref{alg_gen}.

\begin{proposition}
 \label{pr:special}
Let $\varphi_k(\cdot)$ and $p^k$ be as \eqref{eq:def.phi2} and \eqref{eq:def.pk}, respectively, and let the nonempty closed and convex set $\mathcal{S}$ be as in \eqref{eq:def.S}. 
 The following holds for all $k\geq 0$:
 \begin{itemize}
\item[\emph{(a)}] We have
\begin{align}
 \label{eq:browder2}
 \nabla \varphi_k = \left(\gamma^{-1}v^k, u_1^k, u_2^k,\dots, u_{n-1}^k\right)\;\;\mbox{and}\;\; 
 \norm{\nabla \varphi_k}_\gamma^2 = \pi_k
\end{align}
where $\nabla \varphi_k$ is taken with respect to the inner product $\inner{\cdot}{\cdot}_\gamma$ as in \eqref{eq:def.inner}.
\item[\emph{(b)}] For all $p\in \HHH$,
\begin{align}
 \label{eq:car.phi}
 \varphi_k (p) = \sum_{i=1}^n\,\left[\inner{G_i z - x_i^k}{y_i^k - w_i}-\dfrac{1}{4\beta_i}\norm{x_i^k - G_i z^k}^2\right],
\end{align}
where $w_n$ is as in \eqref{eq:def.wn}.
\item[\emph{(c)}] $\varphi_k(\cdot)$ is an affine function in $\HHH$ and
$\mathcal{S}\subset \{p\in \HHH\;|\;\varphi_k(p)\leq 0\}$.
\item[\emph{(d)}] We have
   \begin{align}
    \label{eq:minty}
   p^{k+1}= p^{k} - \tau_k \dfrac{\varphi_k(p^{k})}{\norm{\nabla \varphi_k}_\gamma^2}\nabla \varphi_k\;\;\;
   \mbox{if}\;\;\; \varphi_k(p^k)>0\;\;\; \mbox{else}\;\;\;  p^{k+1}=p^k.
   \end{align}
\item[\emph{(e)}] \emph{Algorithm \ref{alg_main}} is a special instance
   of \emph{Algorithm \ref{alg_gen}} for finding a point in $\mathcal{S}$ as in \eqref{eq:def.S}.
\item[\emph{(f)}] The sequences $(z^k)$ and $(w_i^k)$ \emph{(}$i=1,\dots, n$\emph{)} are bounded.
\end{itemize}

\end{proposition}
\begin{proof}
(a) This follows directly from \eqref{eq:def.phi2}, lines \ref{ln:line501}--\ref{ln:line503} of Algorithm \ref{alg_main} and the definitions
of $\inner{\cdot}{\cdot}_\gamma$ and $\|\cdot\|_\gamma$ in \eqref{eq:def.inner}.

(b) This follows from \eqref{eq:def.phi2}, lines \ref{ln:line501} and \ref{ln:line502} of Algorithm \ref{alg_main}, \eqref{eq:def.wn}
and some simple algebraic manipulations.

(c) The fact that $\varphi_k(\cdot)$ is affine also follows directly from \eqref{eq:def.phi2}. 
To finish the proof of (c), take any $p=(z,w_1,\dots, w_{n-1})\in \mathcal{S}$ and let $w_n$ be as in \eqref{eq:def.wn}. 
From the inclusions in the definition of $\mathcal{S}$ we obtain 
$w_i - (B_i+C_i+D_i)(G_iz)\in A_i(G_i z)$ ($i=1,\dots, n$). Hence, using
Lemma \ref{lm:fund} with $(r_i, a_i)= (G_i z, w_i - (B_i+C_i+D_i)(G_iz))$ ($i=1,\dots, n$) we find
\begin{align*} 
     \inner{G_i z-x_i^k}{y_i^k - w_i} - \dfrac{1}{4\beta_i}\norm{x_i^k - G_i z^k}^2\leq 0\qquad (i=1,\dots, n).
    \end{align*}
The fact that $\varphi_k(p)\leq 0$ for all $p\in \mathcal{S}$ now follows by
adding the above inequalities from $i=1$ to $i=n$ and using \eqref{eq:car.phi}.

(d) Using \eqref{eq:def.phi2}, \eqref{eq:def.pk}, \eqref{eq:browder2} and 
lines \ref{ln:line700}, \ref{ln:line506}, \ref{ln:line504}, \ref{ln:line505}, \ref{ln:line500} and \ref{ln:line505b}
of Algorithm \ref{alg_main} we obtain \eqref{eq:minty}. 

(e) This follows from items (c), (d) and Algorithm \ref{alg_gen}'s definition.

(f) This is a consequence of item (e), Lemma \ref{lm:alg_gen}(a), \eqref{eq:def.pk} and \eqref{eq:def.wn}.
\end{proof}

\mgap

From now on we also assume the following regarding the parameter $\rho_i^k>0$ as in line \ref{ln:line514}
of Algorithm \ref{alg_main}:
\begin{itemize}
 \item[(A8)] For $i\in \{1,\dots,n\}\backslash\mathcal{I}_{\mathcal{D}}$,
   \begin{align}
    \label{eq:conway}
   0<\underline{\rho}_{\,i}\leq \rho_i^k\leq \overline{\rho}_i< \dfrac{1}{\dfrac{1}{4\beta_i}+\ell_i}\qquad \forall k\geq 0,
  \end{align}
where $\ell_i\in [0,\infty)$ and $\beta_i\in (0,\infty]$ are as in Assumptions (A3) and (A4), respectively. 
\end{itemize}

\mgap

\begin{lemma}
 \label{lm:galeao02}
  Let $(x_i^k)$, $(y_i^k)$ and $(\rho_i^k)$ \emph{(}$1\leq i\leq n$\emph{)} be generated by \emph{Algorithm \ref{alg_main}}.
  The following holds for $i=1,\dots, n$:
\begin{itemize}
\item[\emph{(a)}] The sequence $(\rho_i^k)$ is bounded.
\item[\emph{(b)}] The sequence $(x_i^k)$ is bounded.
\item[\emph{(c)}] There exist $\rho_{\emph{min}},\rho_{\emph{max}}>0$ such that
$0<\rho_{\emph{min}}\leq \rho_i^k\leq \rho_{\emph{max}}<\infty$ \quad $\forall k\geq 0$.
\item[\emph{(d)}] The sequence $(y_i^k)$ is bounded.
\end{itemize}
\end{lemma}
\begin{proof}
(a) If $(\rho_i^k,x_i^k,y_i^k)$ is as in line \ref{ln:line600} of Algorithm \ref{alg_main}, then $\rho_i^k=\hat \rho$.
On the other hand, if $(\rho_i^k,x_i^k,y_i^k)$ is as in lines \ref{ln:line508}--\ref{ln:line510}, then
using the second inequality in line \ref{ln:line510} we obtain, $\rho_i^k\leq \overline{\theta}/\hat \delta$.
Now, if $i\in \{1,\dots, n\}\backslash \mathcal{I}_\mathcal{D}$, then it follows from \eqref{eq:conway}
that $\rho_i^k\leq \overline{\rho}:=\max\,\overline{\rho}_i$.
Altogether, we proved that, for $i=1,\dots, n$,
\begin{align}
  \label{eq:galeao03}
 0<\rho_i^k\leq \rho_{\mbox{max}}:=\max\left\{\hat\rho, \dfrac{\overline{\theta}}{\hat \delta},\overline{\rho}\right\}
\qquad \forall k\geq 0.
\end{align}

(b) Take $i\in \{1,\dots, n\}$ and let $u^k:=G_i z^k$ and $s^k:=G_iz^k + \rho_i^k w_i^k - \rho_i^k 
     (B_i+C_i)(G_i z^k)$.
From Proposition \ref{pr:special}(f), the (Lipschitz) continuity of $B_i$, $C_i$
and $G_i$, and \eqref{eq:galeao03} we obtain that $(u^k)$ and $(s^k)$ are bounded sequences.
Using item (a), Lemma \ref{lm:bound} with $(T,\HH,D,m):=(A_i,\HH_i,D_i,m_i)$ and $(u^k)$ and $(s^k)$ as above, we conclude that $(x_i^k)$ as in line \ref{ln:line508} of Algorithm \ref{alg_main} is bounded. 
On the other hand, using item (a) and Lemma \ref{lm:bound} with $(T,\HH,D,m):=(A_i,\HH_i,0,0)$, 
we find that $(x_i^k)$ as in line \ref{ln:line512} of Algorithm \ref{alg_main} is bounded as well. 
Moreover, $(x_i^k)$ as in line
\ref{ln:line600} of Algorithm \ref{alg_main} is trivially bounded by Proposition \ref{pr:special}(f) and the boundedness of
$G_i$.

(c) By \eqref{eq:galeao03}, it remains to prove that there exists $\rho_{\mbox{min}}>0$ such that 
$\rho_i^k\geq \rho_{\mbox{min}}>0$ ($i=1,\dots, n$) for all $k\geq 0$. 
If $i\in \{1,\dots, n\}\backslash \mathcal{I}_{\mathcal{D}}$, then it follows from \eqref{eq:conway} that
$\rho_i^k\geq \underline{\rho}:=\min\,\underline{\rho}_{\,i}>0$ for all $k\geq 0$. 
If $(\rho_i^k,x_i^k,y_i^k)$ is as in line \ref{ln:line600} of Algorithm \ref{alg_main}, then $\rho_i^k=\hat \rho$.
For $(\rho_i^k,x_i^k,y_i^k)$ is as in lines \ref{ln:line508}--\ref{ln:line510}, 
let $c_i^2:=m_i^2\,\sup_{k}\,\norm{x_i^k-G_i z^k}^2$ and note that
from the first inequality in  line \ref{ln:line510} we obtain 
$(4\ell_i^2 + c_i^2)(\rho_i^k)^2+\left(\beta_i^{-1}+\hat \delta\right)\rho_i^k - \underline{\theta}\geq 0$, which gives
$\rho_i^k \geq 2 \underline{\theta}/\Big(\beta_i^{-1}+\hat \delta+\sqrt{(\beta_i^{-1}+\hat \delta)^2+
4(4\ell_i^2+c_i^2)\underline{\theta}}\;\Big)$.
Altogether, we proved that, for $k\geq 0$,
\begin{align}
  \label{eq:galeao033}
  \rho_i^k\geq \rho_{\mbox{min}}:=
      \min\left\{\underline{\rho},\hat\rho,\dfrac{2\underline{\theta}}
{\max \left\{\beta_i^{-1}+\hat \delta+\sqrt{(\beta_i^{-1}+\hat \delta)^2+4(4\ell_i^2+c_i^2)
\underline{\theta}}\right\}_{i\in \mathcal{I}_{\mathcal{D}}}}\right\}>0,
\end{align}
which yields the desired lower bound on the sequence $(\rho_i^k)$ ($i=1,\cdots, n$).

(d) The boundedness of $(y_i^k)$ (see lines \ref{ln:line600}, \ref{ln:line509} and \ref{ln:line513} of Algorithm \ref{alg_main}) follows from Proposition \ref{pr:special}(f), items (b) and (c) above, the (Lipschitz) continuity of $G_i$, $B_i$ and 
Lemma \ref{lm:lip}.
\end{proof}

\mgap

\begin{lemma}
 \label{lm:dantzig}
 Consider the sequences evolved by \emph{Algorithm \ref{alg_main}}. For 
$i\in \{1,\dots,n\}\backslash\mathcal{I}_{\mathcal{D}}$,
 \begin{align}
   \label{eq:carmo}
  (1-\rho_i^k \ell_i)\norm{x_i^k-G_i z^k}\leq \norm{\rho_i^k(y_i^k-w_i^k)}\leq (1+\rho_i^k \ell_i)\norm{x_i^k-G_i z^k}.
 \end{align}
 \end{lemma}
\begin{proof}
First note that if $(\rho_i^k, x_i^k, y_i^k)$ is as in line \ref{ln:line600} of Algorithm \ref{alg_main}, then \eqref{eq:carmo} holds trivially.
On the other hand, using line \ref{ln:line513} of Algorithm \ref{alg_main} and Assumption (A3), we find
\begin{align}
\nonumber
  \norm{x_i^k-G_iz^k+\rho_i^k(y_i^k-w_i^k)}& = \norm{\rho_i^k(B_i(x_i^k)-B_i(G_iz^k))}\\
 \label{eq:tucker}
      &\leq \rho_i^k \ell_i\norm{x_i^k-G_i z^k},
\end{align}
which after using the triangle inequality and some algebraic manipulations gives the desired result.
\end{proof}

\mgap

Next proposition will be important in the convergence proof of Algorithm \ref{alg_main}.

\mgap

\begin{proposition}
 \label{pr:geq}
Consider the sequences evolved by 
 \emph{Algorithm \ref{alg_main}}, let $\varphi_k(\cdot)$ and $p^k$ be as in \eqref{eq:def.phi2} and 
 \eqref{eq:def.pk}, respectively, let $\rho_{\emph{min}},\rho_{\emph{max}}>0$ be as in 
 \emph{Lemma \ref{lm:galeao02}(c)} and let 
 $\underline{\rho}:=\min\,\underline{\rho}_{\,i}$ and 
$\overline{\rho}:=\max\,\overline{\rho}_{\,i}$, where 
$\underline{\rho}_{\,i},\overline{\rho}_{\,i}>0$ \emph{(}$1\leq i\leq n$\emph{)} are as in \eqref{eq:conway}. 
Define, for $i=1,\dots, n$,  
\begin{align}
 \label{eq:def.Delta}
 \Delta_i^k:=\inner{G_i z^k - x_i^k}{y_i^k-w_i^k} - \dfrac{1}{4\beta_i}\norm{x_i^k-G_i z^k}^2.
\end{align}
The following holds:
\begin{itemize}
\item[\emph{(a)}] For $i\in \{1,\dots, n\}\backslash \mathcal{I}_{\mathcal{D}}$,
\begin{align}
 \label{eq:lete}
  \Delta_i^k\geq 
     \dfrac{1}{2}\left[\dfrac{1}{\overline{\rho}}-\left(\dfrac{1}{4\beta_i}+\ell_i\right)\right]
      \left(\norm{x_i^k-G_i z^k}^2+\left(\dfrac{\underline{\rho}}{1+\overline{\rho}\ell_i}\right)^2
         \norm{y_i^k-w_i^k}^2\right) \quad \forall k\geq 0.
\end{align}
\item[\emph{(b)}] For $i\in \mathcal{I}_{\mathcal{D}}$,
\begin{align}
  \label{eq:lete02}
\Delta_i^k\geq
 \dfrac{(2 - \overline{\theta})}{4\rho_{\emph{max}}}\norm{x_i^k-G_iz^k}^2 +\dfrac{\rho_{\emph{min}}}{2}\norm{y_i^k-w_i^k}^2
\qquad \forall k\geq 0,
\end{align}
where $0<\overline{\theta}<2$ is as in the input of Algorithm \ref{alg_main}.
\item[\emph{(c)}] There exist $c_1,c_2>0$ such that, for $i=1,\dots, n$,
\begin{align}
 \label{eq:lete03}
\Delta_i^k\geq
 c_1\norm{x_i^k-G_iz^k}^2 + c_2\norm{y_i^k-w_i^k}^2 \qquad \forall k\geq 0.
\end{align}
\item[\emph{(d)}] There exists $c >0$ such that
\begin{align}
 \label{eq:1419}
  \varphi_k(p^k)\geq 
 c \sum_{i=1}^n\,\left[\norm{x_i^k-G_i z^k}^2+\norm{y_i^k-w_i^k}^2\right]
\qquad \forall k\geq 0.
\end{align}
\end{itemize}
\end{proposition}
\begin{proof}
First note that from the identity $\inner{a}{b}=(1/2)(\norm{a+b}^2-\norm{a}^2-\norm{b}^2)$  we find
\begin{align}
\nonumber
\inner{\underbrace{x_i^k-G_iz^k}_{a}}{\underbrace{\rho_i^k(y_i^k-w_i^k)}_{b}} + 
\dfrac{\rho_i^k}{4\beta_i}\norm{x_i^k-G_iz^k}^2& =
\dfrac{1}{2}\norm{x_i^k-G_iz^k+\rho_i^k(y_i^k-w_i^k)}^2\\
\label{eq:djairo}
 &\hspace{0.1cm} + \left(\dfrac{\rho_i^k}{4\beta_i} - \dfrac{1}{2}\right)\norm{x_i^k-G_iz^k}^2 - \dfrac{1}{2}\norm{\rho_i^k(y_i^k-w_i^k)}^2.
\end{align}

(a) If $(\rho_i^k, x_i^k, y_i^k)$ is as in line \ref{ln:line600} of Algorithm \ref{alg_main}, then the result is trivial.
Take $i\in \{1,\dots, n\}\backslash \mathcal{I}_{\mathcal{D}}$. From line \ref{ln:line513} of Algorithm \ref{alg_main}, \eqref{eq:def.Delta}, \eqref{eq:djairo}, Assumption (A3) and the first inequality in \eqref{eq:carmo},
\begin{align*}
 \nonumber
 -\rho_i^k \Delta_i^k &= \dfrac{1}{2}\norm{\rho_i^k [B_i(x_i^k)-B_i(G_iz^k)]}^2
+ \left(\dfrac{\rho_i^k}{4\beta_i} - \dfrac{1}{2}\right)\norm{x_i^k-G_iz^k}^2 
 - \dfrac{1}{2}\norm{\rho_i^k(y_i^k-w_i^k)}^2\\
\nonumber
  & \leq \left[\dfrac{(\rho_i^k \ell_i )^2}{2}
+ \left(\dfrac{\rho_i^k}{4\beta_i} - \dfrac{1}{2}\right)\right]\norm{x_i^k-G_iz^k}^2 
 - \dfrac{1}{2}\norm{\rho_i^k(y_i^k-w_i^k)}^2\qquad \mbox{[by Assumption (A3)]}\\
 \nonumber
 & \leq \left[\dfrac{(\rho_i^k \ell_i )^2}{2}
+ \left(\dfrac{\rho_i^k}{4\beta_i} - \dfrac{1}{2}\right)\right]\norm{x_i^k-G_iz^k}^2 
 - \dfrac{(1-\rho_i^k\ell_i)^2}{2}\norm{x_i^k-G_iz^k}^2\qquad \mbox{[by \eqref{eq:conway} and \eqref{eq:carmo}]}\\
&= \left[\rho_i^k\left(\dfrac{1}{4\beta_i}+\ell_i\right) -1\right]\norm{x_i^k-G_iz^k}^2,
\end{align*}
which in turn (after division by $-\rho_i^k$) yields
\begin{align}
 \nonumber
 \Delta_i^k&\geq \left[\dfrac{1}{\rho_i^k}-\left(\dfrac{1}{4\beta_i}+\ell_i\right)\right]\norm{x_i^k-G_iz^k}^2\\
 \label{eq:minion}
    &\geq  \left[\dfrac{1}{\overline{\rho}}-\left(\dfrac{1}{4\beta_i}+\ell_i\right)\right]\norm{x_i^k-G_iz^k}^2.
\end{align}
On the other hand, using Assumption (A8),  \eqref{eq:minion} combined with the second inequality in \eqref{eq:carmo} gives
\begin{align}
 \label{eq:minion02}
 \Delta_i^k\geq \left[\dfrac{1}{\overline{\rho}}-\left(\dfrac{1}{4\beta_i}+\ell_i\right)\right]
 \left(\dfrac{\underline{\rho}}{1+\overline{\rho}\ell_i}\right)^2\norm{y_i^k-w_i^k}^2.
\end{align}
Note now that \eqref{eq:lete} follows by adding the inequalities in \eqref{eq:minion} and \eqref{eq:minion02}.

(b) Note first that if $(\rho_i^k,x_i^k,y_i^k)$ is as in line \ref{ln:line600}, then, in particular, 
$x_i^k=G_iz^k$ and $y_i^k=w_i^k$, which gives that \eqref{eq:lete02} 
follows trivially from \eqref{eq:def.Delta} (in this case, both sides of \eqref{eq:lete02} are equal to zero).
In what follows in this proof we assume that $(\rho_i^k,x_i^k,y_i^k)$ is as in lines \ref{ln:line508}--\ref{ln:line510} of Algorithm \ref{alg_main}.

Note that using line \ref{ln:line509} of Algorithm \ref{alg_main}, the inequality 
$\dfrac{1}{2}\norm{a+b}^2\leq \norm{a}^2+\norm{b}^2$, Assumptions (A3) and (A5), and 
Lemma \ref{lm:lip} we obtain
\begin{align}
\nonumber
 \dfrac{1}{2}\norm{x_i^k-G_iz^k+\rho_i^k(y_i^k-w_i^k)}^2 & 
       = \dfrac{1}{2}(\rho_i^k)^2 \| [B_i(x_i^k)-B_i(G_i z^k)] + [D_i(x_i^k) - 
D_{i\,(G_i z^k)}(x_i^k)]\|^2\\
\nonumber
 &\leq (\rho_i^k)^2\left(\norm{B_i(x_i^k)-B_i(G_i z^k)}^2 + \norm{D_i(x_i^k) - D_{i\,(G_i z^k)}(x_i^k)}^2\right)\\
\nonumber
 &\leq (\rho_i^k)^2\left(\ell_i^2\norm{x_i^k-G_iz^k}^2+\dfrac{m_i^2}{4}\norm{x_i^k-G_iz^k}^4\right)\\
 \label{eq:djairo02}
 & = \left(\ell_i^2(\rho_i^k)^2+\dfrac{m_i^2(\rho_i^k)^2}{4}\norm{x_i^k-G_iz^k}^2\right)\norm{x_i^k-G_iz^k}^2.
\end{align}

From \eqref{eq:djairo}, \eqref{eq:djairo02}, line \ref{ln:line510} of 
Algorithm \ref{alg_main} (discarding the term $\hat\delta>0$) and \eqref{eq:def.Delta},
\begin{align}
\nonumber
-\rho_i^k\Delta_i^k
\nonumber
& \leq \left[\ell_i^2(\rho_i^k)^2+\dfrac{m_i^2(\rho_i^k)^2}{4}\norm{x_i^k-G_iz^k}^2
 +\left(\dfrac{\rho_i^k}{4\beta_i}-\dfrac{1}{2}\right)\right]\norm{x_i^k-G_iz^k}^2\\
\nonumber
&\hspace{1cm} -\dfrac{1}{2}\norm{\rho_i^k(y_i^k-w_i^k)}^2\\
\nonumber
& = \dfrac{1}{4}\left[4\ell_i^2(\rho_i^k)^2+ \beta_i^{-1}\rho_i^k + m_i^2 (\rho_i^k)^2 
\norm{x_i^k - G_i z^k}^2-2\right]\norm{x_i^k-G_iz^k}^2\\
\nonumber
&\hspace{1cm} -\dfrac{1}{2}\norm{\rho_i^k(y_i^k-w_i^k)}^2\\
\nonumber
&\leq -\dfrac{2-\overline{\theta}}{4}\norm{x_i^k-G_iz^k}^2 -\dfrac{1}{2}\norm{\rho_i^k(y_i^k-w_i^k)}^2.
\end{align}
The desired result follows by dividing the latter inequality by $-\rho_i^k$
and using Lemma \ref{lm:galeao02}(c).

(c) This follows trivially from items (a) and (b).

(d) This follows from (c), \eqref{eq:def.Delta} and Proposition \ref{pr:special}(b).
\end{proof}

\mgap

\begin{lemma}
  \label{lm:previous}
Consider the sequences evolved by \emph{Algorithm \ref{alg_main}} and let $\varphi_k(\cdot)$ and $p^k$ be as in \eqref{eq:def.phi2}
and \eqref{eq:def.pk}, respectively. The following holds:
\begin{itemize}
\item[\emph{(a)}] There exists $\xi>0$ such that $\norm{\nabla \varphi_k}_\gamma\leq \xi$ for all $k\geq 0$.
\item[\emph{(b)}] $G_i z^k - x_i^k \to 0$ and $w_i^k-y_i^k\to 0$\; \emph{(}$i=1,\dots, n$\emph{)}.
\item[\emph{(c)}] $\sum_{i=1}^n\,G_i^* y_i^k\to 0$ and $x_i^k-G_i x_n^k\to 0$\; \emph{(}$i=1,\dots, n-1$\emph{)}.
\end{itemize}
\end{lemma}
\begin{proof}
(a) From \eqref{eq:browder2}, \eqref{eq:def.inner} and lines \ref{ln:line501}--\ref{ln:line503} of Algorithm \ref{alg_main},
\begin{align*}
\norm{\nabla \varphi_k}_\gamma^2 &= \gamma^{-1}\norm{v^k}^2 + \sum_{i=1}^{n-1}\norm{u_i^k}^2\\
   &=\gamma^{-1}\left\|\sum_{i=1}^n\,G_i^* y_i^k \right\|^2 + \sum_{i=1}^{n-1}\norm{x_i^k-G_i x_n^k}^2\\
   &\leq n\gamma^{-1}\sum_{i=1}^n\,\norm{G_i^*}^2\norm{y_i^k}^2 + 2\sum_{i=1}^{n-1}\, \left(\norm{x_i^k}^2+\norm{G_i}^2    
   \norm{x_n^k}^2\right),
\end{align*}
which combined with Lemma \ref{lm:galeao02}(items (b) and (d)) yields the desired result.

(b) Using Proposition \ref{pr:special}(e), item (a) and Lemma \ref{lm:alg_gen}(c) we obtain $\overline{\lim}\varphi_k(p^k)\leq 0$, which,
in turn, combined with Proposition \ref{pr:geq}(d) yields the desired result.

(c) Using the fact that $\sum_{i=1}^n G_i^* w_i^k=0$ (see line \ref{ln:line511} in Algorithm \ref{alg_main}) we obtain
\begin{align}
  \label{eq:1355}
  \left\|\sum_{i=1}^n\,G_i^* y_i^k\right\| = \left\|\sum_{i=1}^n\,G_i^*(w_i^k - y_i^k)\right\|\leq \sum_{i=1}^n\,
  \norm{G_i^*}\norm{w_i^k-y_i^k}.
\end{align}
Moreover,
\begin{align}
 \label{eq:1312}
  \nonumber
 \norm{x_i^k-G_i x_n^k}&= \norm{x_i^k - G_iz^k + G_i(z^k- x_n^k)}\\
   &\leq \norm{x_i^k - G_iz^k}+\norm{G_i}\norm{z^k- x_n^k} .     
\end{align}
The desired result now follows form \eqref{eq:1355}, \eqref{eq:1312}, item (b) and the fact that $G_n=I$ (see Assumption (A1)).
\end{proof}

\mgap

\begin{theorem}[Weak convergence of Algorithm \ref{alg_main}]
 \label{th:main}
  Consider the sequences evolved by \emph{Algorithm \ref{alg_main}}, let $\mathcal{S}$ be as in \eqref{eq:def.S}
  and assume that Assumptions \emph{(A1)-(A8)} hold.
   Let $p^k:=(z^k, w_1^k,\dots, w_{n-1}^k)$ and let, for all $k\geq 0$,  
\begin{align}
  \label{eq:def.wn5}
   w_n^k:=-\sum_{i=1}^{n-1}\,G_i^* w_i^{k}.
   \end{align}
  Then $(p^k)$ converges weakly to some element 
  $\overline{p}=(\overline{z}, \overline{w}_1,\dots \overline{w}_{n-1})\in \mathcal{S}$, i.e.,
\begin{align}
z^k \rightharpoonup \overline{z}\;\;\mbox{and}\;\;w_i^k\rightharpoonup \overline{w}_i\quad (i=1,\dots, n),
\end{align}
where $\overline{w}_n:=-\sum_{i=1}^{n-1}\,G_i^*\overline{w}_i$,
and
\begin{align}
 \overline{w}_i\in T_i(G_i\overline{z})\qquad (i=1,\dots, n).
\end{align}
Moreover, 
\begin{align}
 x_i^k \rightharpoonup G_i\overline{z}\;\;\mbox{and}\;\; y_i^k\rightharpoonup \overline{w}_i\qquad (i=1,\dots, n).
\end{align}
\end{theorem}
\begin{proof}
We will apply Proposition \ref{pr:special}(e) and Lemma \ref{lm:alg_gen}(b).
Hence, we have to show that every weak cluster point of $p^k=(z^k, w_1^k,\dots, w_{n-1}^k)$ 
belongs to $\mathcal{S}$.
Let $p^\infty=(z^\infty,w_1^\infty, \dots, w_{n-1}^\infty)$ be a weak cluster point of $(p^k)$
and define
 \begin{align}
   \label{eq:def.wn2}
    w^\infty = (w_1^\infty,\dots, w_{n-1}^\infty)\;\;\mbox{and}\;\;
     w_n^\infty=-\sum_{i=1}^{n-1}\,G_i^*w_i^\infty.
 \end{align}
Our goal is to show that $p^\infty\in \mathcal{S}$, i.e.,
\begin{align}
  \label{eq:marina}
 w_i^\infty\in T_i(G_i z^\infty)\qquad (i=1,\dots, n).
\end{align}
(recall that $G_n=I$). To this end, first consider $R:\HH_n=\HH_0 \tos \HH_n=\HH_0$ and 
$S:\HH_1\times \dots \times \HH_{n-1}\tos \HH_1\times \dots \times \HH_{n-1}$ the maximal monotone operators defined as
\begin{align}
 \label{eq:def.ab}
 R = T_n\;\;\mbox{and}\;\; S=T_1\times \dots \times T_{n-1}.
\end{align}
Consider also the (bounded) linear operator $G:\HH_n=\HH_0\to \HH_1\times \dots \times \HH_{n-1}$ defined by
\begin{align}
 \label{eq:def.g}
 G z= (G_1 z, \dots, G_{n-1} z)\qquad (z\in \HH_n=\HH_0)
\end{align}
and let $G^*:\HH_1\times \dots \times \HH_{n-1}\to \HH_n=\HH_0$, 
\begin{align}
  \label{eq:adj.g}
  G^*(w_1,\dots, w_{n-1})= \sum_{i=1}^{n-1}\,G_i^*w_i\qquad \left((w_1,\dots, w_{n-1})\in \HH_1\times \dots \times \HH_{n-1}\right)
\end{align}
be its adjoint operator.
In view of \eqref{eq:def.ab}, \eqref{eq:def.wn2}, \eqref{eq:adj.g} and the fact that $G_n=I$, we have that
our goal \eqref{eq:marina} is clearly equivalent to 
\begin{align}
 \label{eq:marina02}
  -G^*w^\infty\in R(z^\infty)\;\;\mbox{and}\;\; w^\infty \in S(Gz^\infty)\;\;\mbox{i.e.}\;\;
   (-G^*w^\infty, w^\infty)\in (R\times S)(z^\infty, Gz^\infty).
\end{align}
Since $R\times S$ is maximal monotone~\cite{bau.com-book}, to prove \eqref{eq:marina02} (and hence our goal \eqref{eq:marina}) it suffices
to show that
\begin{align}
 \label{eq:marina03}
 \inner{(r,s)-(z^\infty, Gz^\infty)}{(u,v)-(-G^*w^\infty, w^\infty)}\geq 0\qquad \forall (u,v)\in (R\times S)(r,s).
\end{align}
Hence, in what follows in this proof,  we will focus on proving that \eqref{eq:marina03} holds. 
Let us start by proving that
\begin{align}
 \label{eq:neumann3}
\inner{(r,s)-(r^k, s^k)}{(u,v)-(u^k, v^k)} + 
 \sum_{i=1}^{n}\,\dfrac{1}{4\beta_i}\norm{x_i^k - G_i z^k}^2\geq 0\qquad \forall (u, v)\in (R\times S)(r,s),
\end{align}
where, for all $k\geq 0$,
\begin{align}
 \label{eq:def.rs}
r^k:=x_n^k,\quad  u^k:=y_n^k,\quad s^k:=(x_1^k,\dots, x_{n-1}^k)\;\;\mbox{and}\;\; v^k:=(y_1^k,\dots, y_{n-1}^k).
\end{align}
To this end, take 
\begin{align}
 \label{eq:rock22}
 (u, v)\in (R\times S)(r,s),\;\;\mbox{i.e.}\;\; u\in R(r)\;\;\mbox{and}\;\; v\in S(s),
\end{align}
where $r,u\in \HH_0$ and $s:=(s_1,\dots, s_{n-1})$, $v:=(v_1,\dots, v_{n-1})\in \HH_1\times \dots \times \HH_{n-1}$.
Using the latter definitions and \eqref{eq:def.rs} we obtain, for all $k\geq 0$,
\begin{align}
 \label{eq:1131}
  \nonumber
\inner{(r,s)-(r^k, s^k)}{(u,v)-(u^k, v^k)} &= \inner{r-r^k}{u-u^k}+\inner{s-s^k}{v-v^k}\\
    & = \inner{r-x_n^k}{u-y_n^k} +  \sum_{i=1}^{n-1}\,\inner{s_i -x_i^k}{v_i - y_i^k}.
\end{align}
In view of \eqref{eq:rock22}, \eqref{eq:def.ab} and the fact that $T_i=A_i+B_i+C_i+D_i$ ($i=1,\dots, n$) we have
\begin{align*}
 u- B_n(r)-C_n(r)-D_n(r) \in A_n(r)\;\;\mbox{and}\;\;v_i - B_i(s_i)-C_i(s_i)-D_i(s_i)\in A_i(s_i) \quad (i=1,\dots, n-1).
\end{align*}
Hence, using Lemma \ref{lm:fund} with $(r_i, a_i):=(s_i, v_i - B_i(s_i)-C_i(s_i)-D_i(s_i))$  ($i=1,\dots, n-1$) 
and $(r_n,a_n):=(r,u-B_n(r)-C_n(r)-D_n(r))$ and some simple algebra we
find
\begin{align*}
 \inner{r-x_n^k}{u - y_n^k}\geq -\dfrac{1}{4\beta_n}\norm{x_n^k-G_n z^k}^2\;\;\mbox{and}\;\;
 \sum_{i=1}^{n-1}\,\inner{s_i -x_i^k}{v_i - y_i^k}\geq -\sum_{i=1}^{n-1}\,\dfrac{1}{4\beta_i}\norm{x_i^k-G_i z^k}^2,
\end{align*}
which, in turn, when combined with \eqref{eq:1131} yields exactly \eqref{eq:neumann3}.

Next we will take limit in \eqref{eq:neumann3} (up to a subsequence) to prove \eqref{eq:marina03} (and hence our goal \eqref{eq:marina}).
To this end, let $(p^{k_j})$ be a subsequence of $p^k=(z^k, w_1^k,\dots, w_{n-1}^k)$ such that $p^{k_j}\rightharpoonup p^\infty=(z^\infty,w_1^\infty, \dots, w_{n-1}^\infty)$. (Recall that $p^\infty$ is a weak cluster point of $(p^k)$.)
 Using \eqref{eq:def.wn5} and \eqref{eq:def.wn2} combined with the assumption that $p^{k_j}\rightharpoonup p^\infty$  we obtain
\begin{align}
 \label{eq:goldfarb}
 z^{k_j}\rightharpoonup z^{\infty}\;\;\;\mbox{and}\;\;\;w_i^{k_j}\rightharpoonup w_i^\infty\qquad (i=1,\dots, n).
\end{align}
From Lemma \ref{lm:previous}(b) and the first limit in \eqref{eq:goldfarb} 
we obtain $x_i^{k_j}- G_i z^{k_j}\to 0$ and $G_i z^{k_j}\wto G_i z^\infty$, respectively, and so
\begin{align}
 \label{eq:broyden}
  x_i^{k_j} = \left(x_i^{k_j}- G_i z^{k_j}\right) + G_i z^{k_j}    
   \rightharpoonup  G_i z^\infty \qquad (i=1,\dots, n).
\end{align}
Analogously, using now the second limit in \eqref{eq:goldfarb} and Lemma \ref{lm:previous}(b) again we also find
\begin{align}
 \label{eq:broyden2}
  y_i^{k_j} = \left(y_i^{k_j}- w_i^{k_j}\right) + w_i^{k_j}   
   \rightharpoonup  w_i^\infty \qquad (i=1,\dots, n).
\end{align}
Using the definitions of $r^k$ and $s^k$  in \eqref{eq:def.rs} combined 
\eqref{eq:broyden} for $i=n$ and for $i=1,\dots, n-1$ (respectively), the fact that $G_n=I$ and \eqref{eq:def.g} we also find
\begin{align}
 \label{eq:nocedal}
 r^{k_j}\rightharpoonup z^\infty\;\;\mbox{and}\;\; s^{k_j}\rightharpoonup G z^\infty\;\;\mbox{i.e.}\;\; 
 (r^{k_j},s^{k_j})\rightharpoonup (z^\infty,G z^\infty).
\end{align}
Likewise, using now the definitions of $u^k$ and $v^k$ in \eqref{eq:def.rs} combined with \eqref{eq:broyden2} for $i=n$ and for $i=1,\dots, n-1$ (respectively), \eqref{eq:def.wn2} and \eqref{eq:adj.g} we obtain 
\begin{align}
 \label{eq:nocedal2}
 u^{k_j}\rightharpoonup -G^* w^\infty\;\;\mbox{and}\;\;  v^{k_j}\rightharpoonup w^\infty
 \;\;\mbox{i.e.}\;\; 
 (u^{k_j},v^{k_j})\rightharpoonup (-G^* w^\infty, w^\infty).
\end{align}
Note now that, from the definitions of $u^k$ and $v^k$ in \eqref{eq:def.rs}, \eqref{eq:adj.g}
and the fact that $G_n=I$ we have $u^k + G^*  v^k= y_n^k+\sum_{i=1}^{n-1}\,G_i^*y_i^k = \sum_{i=1}^{n}\,G_i^*y_i^k$ and so, from Lemma \ref{lm:previous}(c),
\begin{align}
  \label{eq:uvk}
 u^k + G^*  v^k \to 0.
\end{align}
Analogously, from the definitions of $r^k$ and $s^k$ in \eqref{eq:def.rs} and the definition of $G$ in \eqref{eq:def.g} we also find
$G r^k - s^k = (G_1 x_n^k - x_1^k, \dots, G_{n-1} x_n^k - x_{n-1}^k)$ and so, from 
Lemma \ref{lm:previous}(c), 
\begin{align}\label{eq:rsk}
 G r^k - s^k \to 0.
\end{align}

Now consider $V$ the closed linear subspace of $\HHH$ defined as 
$V=\{(z,Gz)\,| z\in \HH_0\}$, i.e, $V$ is the graph of the linear operator $G$, and let $P=P_V$ and $P^\perp=P_{V^\perp}$ denote
the projection operators onto $V$ and $V^\perp$, respectively. 
From Lemma \ref{lm:pro_nuc} we know that, for all $(z,w)\in \HHH$, with $z\in \HH_0$,
\begin{align}
 \begin{aligned}
& P(z, w) = ((I+G^*G)^{-1}(z+G^*w), G(I+G^*G)^{-1}(z+G^*w)),\\
& P^\perp(z,w) = ( G^*(I+GG^*)^{-1}(Gz-w), -(I+GG^*)^{-1}(Gz-w)).
\end{aligned}
\end{align}
Hence, using \eqref{eq:uvk}, \eqref{eq:rsk} as well as the continuity of $(I+G^*G)^{-1}$, $(I+GG^*)^{-1}$ and $G^*$
we find
\begin{align} \label{eq:neumann}
 P(u^k,v^k)\to 0\;\;\mbox{and}\;\; P^\perp(r^k,s^k) \to 0.
\end{align}
On the other hand, using the fact that $P$ and $P^\perp$ are weakly continuous combined with \eqref{eq:nocedal} and \eqref{eq:nocedal2}
we also find
\begin{align}
 \label{eq:neumann2}
  \begin{aligned}
 & P(r^{k_j}, s^{k_j})\rightharpoonup P(z^\infty, Gz^\infty)=(z^\infty, Gz^\infty),\\
 &P^\perp(u^{k_j},v^{k_j})\rightharpoonup
 P^\perp(-G^* w^\infty, w^\infty)=(-G^* w^\infty, w^\infty),
 \end{aligned}
\end{align}
where the latter identities follow from the fact that $(z^\infty, Gz^\infty)\in V$ and $(-G^* w^\infty, w^\infty)\in V^\perp$.

Now, using the fact that $V$ is orthogonal to $V^{\perp}$ and some algebraic manipulations we find that the first term in the right-hand side of
\eqref{eq:neumann3} can be decomposed as
\begin{align}
\label{eq:newton}
\begin{aligned}
\inner{(r,s)-(r^k, s^k)}{(u,v)-(u^k, v^k)}&= \underbrace{\inner{P(r,s)-P(r^k, s^k)}{P(u,v)-P(u^k, v^k)}}_{=:\mu^k}\\[2mm]
 & + \underbrace{\inner{P^\perp(r,s)-P^\perp(r^k, s^k)}{P^\perp(u,v)-P^\perp(u^k, v^k)}}_{=:\nu^k}.
 \end{aligned}
\end{align}
Direct use of \eqref{eq:neumann} with $k=k_j$ and \eqref{eq:neumann2} yields
\begin{align*}
 \mu^{k_j}  \to \inner{P(r,s)-P(z^\infty, Gz^\infty)}{P(u,v)}
 \;\;\mbox{and}\;\;
 \nu^{k_j}\to \inner{P^\perp(r,s)}{P^\perp(u,v)-P^\perp (-G^* w^\infty, w^\infty)}
\end{align*}
and so
\begin{align} \lab{eq:vera}
\mu^{k_j} + \nu^{k_j}\to 
\inner{P(r,s)-P(z^\infty, Gz^\infty)}{P(u,v)} + \inner{P^\perp(r,s)}{P^\perp(u,v)-P^\perp (-G^* w^\infty, w^\infty)}=:\eta.
\end{align}
On the other hand, using the facts that $V$ is orthogonal to $V^\perp$ and $P+P^\perp=I$ we find that the right-hand side of
\eqref{eq:vera} can be written as 
\begin{align*}
\eta 
     &=\inner{P(r,s)-P(z^\infty, Gz^\infty)+P^\perp(r,s)}{P(u,v)+P^\perp(u,v)- P^\perp(-G^* w^\infty, w^\infty)}\\
     &= \inner{(r,s)-P(z^\infty, Gz^\infty)}{(u,v)- P^\perp(-G^* w^\infty, w^\infty)}\\
     &= \inner{(r,s)-(z^\infty, Gz^\infty)}{(u,v)- (-G^* w^\infty, w^\infty)},
\end{align*}
where in the last identity we used that $(z^\infty, Gz^\infty)\in V$ and $(-G^* w^\infty, w^\infty)\in V^\perp$. Hence, from
\eqref{eq:vera},
\begin{align}
\label{eq:miranda}
 \mu^{k_j}+\nu^{k_j}\to \inner{(r,s)- (z^\infty, Gz^\infty)}{(u,v)- (-G^* w^\infty, w^\infty)}.
\end{align}

Note now that we obtain \eqref{eq:marina03} as a direct consequence of \eqref{eq:neumann3}, \eqref{eq:newton}  with $k=k_j$, \eqref{eq:miranda} and Lemma \ref{lm:previous}(b). 
\end{proof}

\section{Bracketing/bisection procedure}
\lab{sec:bbp}

In this section, we present and justify the bracketing/bisection routine (Algorithm \ref{alg_proc}) 
called in line \ref{ln:line515} of 
Algorithm \ref{alg_main}. 

\mgap

\noindent
{\bf Motivation.} The main motivation for introducing a bracketing/bisection procedure in Algorithm \ref{alg_main} 
is as follows: for each $i\in \mathcal{I}_{\mathcal{D}}$, at an iteration $k\geq 0$, we may need to find $\rho_i^k>0$ and $x_i^k$ (see lines \ref{ln:line508} and \ref{ln:line510} of Algorithm \ref{alg_main}) satisfying
\begin{align}
\label{eq:cha}
& x_i^k = J_{\rho_i^k \left(A_i+D_{i \,(G_i z^k)}\right)} \left(G_iz^k + \rho_i^k w_i^k - \rho_i^k  
            (B_i+C_i)(G_i z^k)\right),\\[2mm]
\label{eq:cha2}
&\underline{\theta}\leq 4\,\ell_i^2 (\rho_i^k)^2 + \left(\beta_i^{-1}+\hat \delta\right)\rho_i^k +
            \left(m_i \rho_i^k \norm{x_i^k - G_i z^k}\right)^2\leq \overline{\theta}.
\end{align}
While the computation of $x_i^k$ as in \eqref{eq:cha} follows essentially from the assumption that the resolvent
$J_{\rho_i^k \left(A_i+D_{i \,(G_i z^k)}\right)}$ is computable, the requirement of $\rho_i^k$ and $x_i^k$ satisfying the
coupled system \eqref{eq:cha}-\eqref{eq:cha2} makes the task much more challenging. 

To begin with, let's put \eqref{eq:cha}-\eqref{eq:cha2} in a more general framework. 
The task then is (cf. \eqref{eq:cha}-\eqref{eq:cha2}):

\mgap

\noindent
\fbox{{\bf Problem A}} 
{\it 
Let $\HH$ be a real Hilbert space, 
$A:\HH\tos \HH$ be maximal monotone, $B:\HH\to \HH$ be monotone and $\ell$-Lipschitz continuous, 
$C:\HH\to \HH$ be $\beta$-cocoercive, $D:\HH\to \HH$ be monotone and continuously differentiable with an $m$-Lipschitz continuous derivative $D'(\cdot)$, and $G:\HH\to \HH$ be a bounded linear operator.  

Given $z,w\in \HH$, find $\rho>0$ and $x\in \HH$ satisfying
\begin{align}
\label{eq:cha3}
& x = J_{\rho \left(A +D_{(G z)}\right)} \left(G z + \rho w - \rho(B+C)(G z)\right),\\[2mm]
\label{eq:cha4}
&\underline{\theta}\leq 4\,\ell^2  \rho^2 + \left(\beta^{-1}+\hat \delta\right)\rho +
             (m\rho\norm{x - G z})^2\leq \overline{\theta},
\end{align}
where $\hat\delta>0$, $0<\underline{\theta}<\overline{\theta}$ and
$D_{(Gz)}(x):=D(Gz) + D'(Gz)(x-Gz)$, for all $x\in \HH$.
}

\mgap

Next lemma and corollary will be helpful.

\begin{lemma}\lab{lm:roseli}
Let $A(\cdot)$, $B(\cdot)$, $C(\cdot)$, $D(\cdot)$ and $G(\cdot)$, and $z,w\in \HH$ and $\rho>0$, be as described in \emph{Problem A} above and 
define $\widehat T:\HH\tos \HH$ as
\begin{align}\lab{eq:that}
\widehat T(x) = A(x) + D_{(Gz)}(x) + (B+C)(Gz) -w,\qquad \forall x\in \HH.
\end{align}
The following holds:
\begin{itemize}
\item[\emph{(a)}] $\widehat T$ is maximal monotone.
\item[\emph{(b)}] $0\in \widehat T(Gz)$ if and only if $w\in (A+B+C+D)(Gz)$.
\item[\emph{(c)}] $J_{\rho\widehat T}(Gz) = J_{\rho(A+D_{(Gz)})}(Gz+\rho w -\rho(B+C)(Gz))$.
\end{itemize}
\end{lemma}
\bproo
(a) The result follows from the definition of $\widehat T$ as in \eqref{eq:that} and the maximality of the sum of maximal monotone operators~\cite{roc-max.tams70} (see also~\cite{bau.com-book}).

(b) This follows trivially from \eqref{eq:that} and \eqref{eq:lin.app}.

(c) This is a direct consequence of \eqref{eq:that} and the definition of resolvents of $\widehat T$ and $A+D_{(Gz)}$.
\eproo

\mgap

\bcoro \lab{cor:cafe}
Let the maximal monotone operator $\widehat T$ be as in \eqref{eq:that}. For $z,w\in \HH$, 
we have that $\rho>0$ and $x\in \HH$ solve \emph{Problem A}, i.e., they satisfy \eqref{eq:cha3}-\eqref{eq:cha4}
if and only if
\begin{align}\lab{eq:cafe}
 \underline{\theta}\leq 4\,\ell^2  \rho^2 + \left(\beta^{-1}+\hat \delta\right)\rho +
             (m\rho\norm{J_{\rho\widehat T}(Gz) - G z})^2\leq \overline{\theta}
\end{align}
and $x=J_{\rho\widehat T}(Gz)$. Consequently, \emph{Problem A} is equivalent to find $\rho>0$ satisfying condition \eqref{eq:cafe}.
\ecoro
\bproo
The proof follows from Lemma \ref{lm:roseli}(c) and \eqref{eq:cha3}-\eqref{eq:cha4}.
\eproo

\mgap

As a direct consequence of Corollary \ref{cor:cafe}, we have that Problem A can be posed in a more general framework as described next . 

\mgap

\noindent
{\bf \fbox{Problem B}}
{\it Given a maximal monotone operator $T$ in a real Hilbert space $\HH$, scalars $a,b\geq 0$, 
$0<\theta_{-}<\theta_{+}<\infty$
and $z\in \HH$, find $\rho>0$ such that
\begin{align}
 \label{eq:sorvete}
 \theta_{-}\leq \psi(\rho, z)\leq \theta_{+},
\end{align}
where 
 \begin{align}
   \label{eq:sorvete2}
 \psi(\rho, z):=a \rho^2 + b\rho + (\rho\norm{J_{\rho T}(z)-z})^2.
 \end{align}
}

\mgap

\begin{proposition}\lab{pr:roseli}
\emph{Problem A} is a special case of \emph{Problem B}. As a consequence, any solution method designed to solve 
\emph{Problem B} can also be applied to solve \emph{Problem A}.
\end{proposition}
\bproo
Indeed, by Corollary \ref{cor:cafe} we know that Problem A is equivalent to the problem of finding $\rho>0$ satisfying
\eqref{eq:cafe}.
Now, note that by simply dividing \eqref{eq:cafe} by $m^2$, we see that it is clearly equivalent to 
\eqref{eq:sorvete}-\eqref{eq:sorvete2} with $T=\widehat T$ and
\begin{align*}
 \theta_{-} \leftarrow \underline{\theta}/m^2,\quad a \leftarrow 4\ell^2/m^2,\quad
 b \leftarrow  (\beta^{-1}+\hat\delta)/m^2,\quad z\leftarrow Gz\;\;\mbox{and}\quad 
 \theta_{+}\leftarrow \overline{\theta}/m^2.
\end{align*}
\eproo

\mgap

\begin{center}
\emph{Based on the discussion in the first paragraph of this section (see ``Motivation'') and Proposition \ref{pr:roseli}, we conclude that it is sufficient for our purposes to consider \emph{Problem B}.}
\end{center}

\mgap

In order to achieve \eqref{eq:sorvete}, we need to study the function $\psi(\cdot, \cdot)$ as in \eqref{eq:sorvete2} in more details. The following lemma of Monteiro and Svaiter will be helpful to us.

\begin{lemma}\emph{(\cite[Lemma 4.3]{mon.sva-newton.siam12})}
 \label{lm:mon.sva}
 Let $\HH$ be a real Hilbert space and $T:\HH\tos \HH$ be maximal monotone.
 The following holds for $\varphi(\rho,z):=\rho\norm{J_{\rho T}(z)-z}$, where $\rho>0$ and $z\in \HH$.
 \begin{itemize}
  \item[\emph{(a)}] $0<\rho\mapsto \varphi(\rho,z)$ is a continuous function.
  \item[\emph{(b)}] For every $0<\mu<\rho$,
  \begin{align}
    \label{eq:1300}
   \dfrac{\rho}{\mu} \varphi(\mu,z)\leq \varphi(\rho,z) \leq \left(\dfrac{\rho}{\mu}\right)^2 \varphi(\mu,z).
  \end{align}
 \end{itemize}
\end{lemma}

\mgap

Next we prove a similar result for the function $\psi(\cdot,\cdot)$.
\mgap

\begin{lemma}
 \label{lm:mon.sva.gen}
 Let $\psi(\cdot,\cdot)$ be as in \eqref{eq:sorvete2}, i.e., let $\psi(\rho,z):= a\rho^2 + b\rho + \left[\varphi(\rho, z)\right]^2$,
 where $\rho>0$, $z\in \HH$, $a,b\geq 0$ and $\varphi(\rho, z)$ is as in \emph{Lemma \ref{lm:mon.sva}}. If $0<\mu<\rho$, then
\begin{align}
 \label{eq:1103}
 \dfrac{\rho}{\mu}\psi(\mu, z)\leq \psi(\rho, z)\leq \left(\dfrac{\rho}{\mu}\right)^4\psi(\mu, z).
\end{align}
\end{lemma}
\begin{proof}
Note first that
\begin{align*}
\dfrac{\psi(\mu, z)}{\mu} - b&= a\mu+\mu\left(\dfrac{\varphi(\mu,z)}{\mu}\right)^2\\
                   & \leq a\rho+\rho\left(\dfrac{\varphi(\mu,z)}{\mu}\right)^2 \qquad [\mu<\rho]\\
                   & \leq a\rho+\rho\left(\dfrac{\varphi(\rho,z)}{\rho}\right)^2\qquad [\mbox{by the first inequality in \eqref{eq:1300}}]\\
                   & = \dfrac{\psi(\rho, z)}{\rho}-b,
\end{align*}
which clearly gives the first inequality in \eqref{eq:1103}.
Note now that it is trivial to check the inequality
\begin{align}
 \label{eq:1305}
  a\rho^2 + b\rho \leq \left(\dfrac{\rho}{\mu}\right)^4 (a\mu^2 + b\mu).
\end{align}
Hence,
\begin{align*}
\psi(\rho,z)&= a\rho^2 + b\rho + \left[\varphi(\rho, z)\right]^2\\
   &\leq \left(\dfrac{\rho}{\mu}\right)^4 \left(a\mu^2 + b\mu\right)+ \left[\varphi(\rho, z)\right]^2\qquad \mbox{[by \eqref{eq:1305}]}\\
   &\leq \left(\dfrac{\rho}{\mu}\right)^4\left( a\mu^2 + b\mu+ \left[\varphi(\mu, z)\right]^2\right)
                \qquad \mbox{[by the second inequality in \eqref{eq:1300}]}\\
                &= \left(\dfrac{\rho}{\mu}\right)^4 \psi(\mu, z).
\end{align*}
\end{proof}

\mgap

\begin{lemma}
 \label{lm:mon.sva.gen02}
 Let $\psi(\cdot, \cdot)$ be as in \eqref{eq:sorvete2} and let $z\in \HH$ be such that 
 $0\notin T(z)$. The following holds:
 \begin{itemize}
 \item[\emph{(a)}] $\psi(\rho, z)>0$ for every $\rho>0$.
 \item[\emph{(b)}] $0<\rho\mapsto \psi(\rho, z)$ is a strictly increasing and continuous function, which goes to $0$ 
 or $\infty$, as
 $\rho$ tends to $0$ or $\infty$, respectively.
\item[\emph{(c)}] For any $0<\theta_{-}<\theta_{+}<\infty$, the set of all $\rho>0$ satisfying
\[
 \theta_{-}\leq \psi(\rho, z)\leq \theta_+
\]
is a closed interval $[\rho_{-},\rho_+]$, where
\begin{align}
 \label{eq:rhom}
  \theta_{-}=\psi(\rho_{-},z)\;\;\mbox{and}\;\; \theta_{+}=\psi(\rho_{+},z),
\end{align}
such that
\[
 \dfrac{\rho_+}{\rho_{-}}\geq \sqrt[4]{\dfrac{\theta_+}{\theta_{-}}}.
\]
 \end{itemize}
\end{lemma}
\begin{proof}
 The proof follows the same outline of \cite[Lemma 4.4]{mon.sva-newton.siam12}'s proof, using Lemma \ref{lm:mon.sva.gen} instead of
 \cite[Lemma 4.3]{mon.sva-newton.siam12}.
\end{proof}

\mgap

Next lemma, which is motivated by \cite[Lemma 4.6]{mon.sva-newton.siam12},  justifies the bracketing/bisection procedure as presented in Algorithm \ref{alg_proc} in the light of Problem B - see \eqref{eq:sorvete} and \eqref{eq:sorvete2}. 

\mgap

\begin{lemma}
 \label{lm:just}
Let $\psi(\cdot, \cdot)$ be as in \eqref{eq:sorvete2} and 
let $\rho_{-},\rho_{+}>0$ be as in \eqref{eq:rhom}.
Let $z\in \HH$ be such that $0\notin T(z)$ and let $0<\theta_{-}<\theta_+<\infty$.
The following holds:
\begin{itemize}
\item[\emph{(a)}] If $\psi(\rho,z)<\theta_{-}$, then $\rho<\rho_{-}$ and $\rho_+\leq \dfrac{\rho\theta_+}{\psi(\rho, z)}$.
\item[\emph{(b)}] If $\psi(\rho,z)>\theta_{+}$, then $\dfrac{\rho\theta_{-}}{\psi(\rho, z)}\leq \rho_{-}$ and $\rho_{+}<\rho$.
\end{itemize}
\end{lemma}
\begin{proof}
(a) If $\psi(\rho,z)<\theta_{-}$, using the fact that $\theta_{-}=\psi(\rho_{-},z)$ - see \eqref{eq:rhom} - and Lemma \ref{lm:mon.sva.gen02}(b) we obtain $\rho<\rho_{-}$. Moreover, since we now know that $\rho<\rho_{-}<\rho_{+}$, one can
use Lemma \ref{lm:mon.sva.gen} with $(\mu,\rho)\leftarrow (\rho,\rho_{+})$ to conclude that
%
%
 $ \dfrac{\rho_{+}}{\rho} \psi(\rho, z)\leq \psi(\rho_{+},z)$
and hence, since $\theta_{+}=\psi(\rho_{+},z)$, that $\rho_+\leq \dfrac{\rho\theta_+}{\psi(\rho, z)}$.

(b) The proof follows the same outline of the proof of item (a).
\end{proof}

 \mgap

\noindent
We now make a few remarks about Lemma \ref{lm:just} (see also Figure \ref{fig01}).
\begin{itemize}
\item[\mbox{(i)}] Lemma \ref{lm:just}(a) says that if \eqref{eq:sorvete} fails, for some $\rho>0$, because $\psi(\rho,z)<\theta_{-}$, then it follows that $t_{-}:=\rho$ and $t_{+}:=\dfrac{\rho\theta_+}{\psi(\rho, z)}$ define a closed interval $[t_{-}, t_{+}]$ which contains the interval $[\rho_{-}, \rho_+]$ of all possible $\rho>0$ satisfying \eqref{eq:sorvete}. A similar observation also applies to Lemma \ref{lm:just}(b). 
This is what we refer to in Algorithm \ref{alg_proc} as {\it Bracketing}.
 \item[\mbox{(ii)}] Since, as we pointed out in item (i) above, $[t_{-}, t_{+}]\supset [\rho_{-}, \rho_{+}]$, a simple bisection procedure on the interval $[t_{-}, t_{+}]$ will finish with some $\rho>0$ in $[\rho_{-}, \rho_{+}]$, i.e., some $\rho>0$ satisfying Problem B. This is what we refer to in Algorithm \ref{alg_proc} as {\it Bisection}.
\end{itemize}

\begin{figure}
 \caption{See Problem B, Lemma \ref{lm:just} and the remarks following it. }
  \label{fig01}
  \centering
  \begin{tikzpicture}
    \draw[->,line width = 0.50mm] (-1.2,0) -- (5.8,0) node[right] {$\rho$};
    \draw[->,line width = 0.50mm] (0,-1.2) -- (0,4.7) node[above] {$\psi(\cdot,z)$};
    \draw[line width = 0.55mm] (0.05,0.05) parabola (5.5,4.5);
    \draw[line width = 0.35mm] (0,0) circle (0.8mm);
    \node[below left,black] at (0,0) {0};
    \node[below left] at (1.38,0) {$t\_$};
    \node[below left] at (5.5,0) {$t_+$};
    \draw[-,line width = 0.50mm] (5.2,-0.05) -- (5.2,0.15);
    \draw[-,line width = 0.50mm] (1.1,-0.05) -- (1.1,0.15);
    \draw [dashed,line width = 0.50mm] (0,0.55) -- (1.9,0.55);
    \draw [dashed,line width = 0.50mm] (1.9,0) -- (1.9,0.55);
    \draw [dashed,line width = 0.50mm] (0,1.5) -- (3.15,1.5);
    \draw [dashed,line width = 0.50mm] (3.15,0) -- (3.15,1.5);
     \node [below left] at (2.3,0) {$\rho\_$};
     \node [below left] at (3.5,0) {$\rho_+$};
     \node [below left] at (2.3,0) {$\rho\_$};
     \node[above left] at (0,0.3) {$\theta\_$};
     \node[above left] at (0,1.2) {$\theta_+$};
\end{tikzpicture} \\
\end{figure}
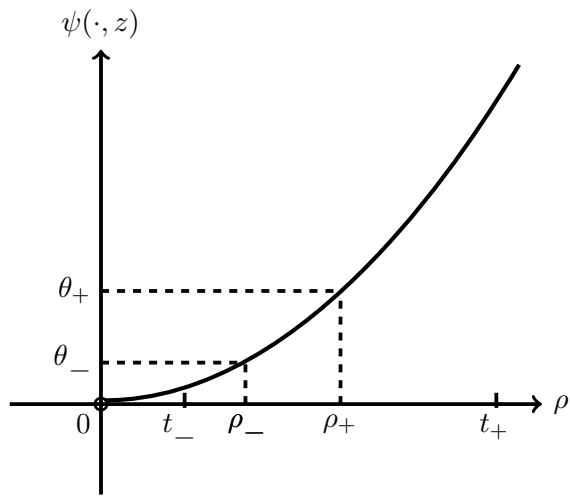

\newpage

\begin{algorithm}[H]
%
 \caption{Bracketing/Bisection procedure for solving Problem A}
\label{alg_proc}
\SetAlgoLined
\KwOutput{$0<\underline{\theta}<\overline{\theta}<2$, $\hat \rho>0$ and $\hat \delta>0$}
   \Fn{\FMain{$A,B,C,D,G,z,w,\ell,\beta, m$}}{
 Choose $\rho^0>0$ and compute $x^0 = J_{\rho^0\left(A+D_{(Gz)}\right)}\left(Gz + \rho^0 w -\rho^0(B+C)(Gz)\right)$\\
   \If{$\underline{\theta}\leq 4\ell^2(\rho^0)^2 + (\beta^{-1}+\hat \delta)\rho^0 +
(m\rho^0\norm{x^0-Gz})^2\leq \overline{\theta}$}
   {\Return $\rho\leftarrow \rho^0$, $x\leftarrow x^0$ and \\
   $y\leftarrow (Gz-x^0)/\rho^0 + w
   +\left[B(x^0) - B(Gz)\right] + \left[D(x^0) - D_{(Gz)}(x^0)\right]$ }
   \If{$4\ell^2(\rho^0)^2 + (\beta^{-1}+\hat \delta)\rho^0 + (m\rho^0\norm{x^0-Gz})^2<\underline{\theta}$}
   {$t_{-}\leftarrow \rho^0$ and $t_{+}\leftarrow \dfrac{\rho^0\overline{\theta}}{4\ell^2(\rho^0)^2 + 
(\beta^{-1}+\hat \delta)\rho^0 + (m\rho^0\norm{x^0-Gz})^2}$}
  \If{$4\ell^2(\rho^0)^2 + (\beta^{-1}+\hat \delta)\rho^0 + (m\rho^0\norm{x^0-Gz})^2>\overline{\theta}$}
   {$t_{-}\leftarrow \dfrac{\rho^0\underline{\theta}}{4\ell^2(\rho^0)^2 + (\beta^{-1}+\hat \delta)\rho^0 + (m\rho^0\norm{x^0-Gz})^2}$ and $t_{+}\leftarrow \rho^0$}
 Set $\tilde \rho = \sqrt{t_{-}t_{+}}$ and        
         compute $\tilde x = J_{\tilde \rho\left(A+D_{(Gz)}\right)}\left(Gz + \tilde \rho w -\tilde \rho(B+C)(Gz)\right)$\\[1mm]
         \label{ln:line_set}
   \If{$\underline{\theta}\leq 4\ell^2(\tilde \rho)^2 + (\beta^{-1}+\hat \delta)\tilde \rho + (m\tilde \rho\norm{\tilde x-Gz})^2\leq \overline{\theta}$}
   {\Return $\rho\leftarrow \tilde \rho$, $x\leftarrow \tilde x$ and \\
    $y\leftarrow (Gz-\tilde x)/\tilde \rho + w
   +\left[B(\tilde x) - B(Gz)\right] + \left[D(\tilde x) - D_{(Gz)}(\tilde x)\right]$ }
   \If{$4\ell^2(\tilde \rho)^2 + (\beta^{-1}+\hat \delta)\tilde \rho + (m\tilde \rho\norm{\tilde x-Gz})^2>  \overline{\theta}$}
   {$t_{+}\leftarrow \tilde \rho$}
   \If{$4\ell^2(\tilde \rho)^2 + (\beta^{-1}+\hat \delta)\tilde \rho + (m\tilde \rho\norm{\tilde x-Gz})^2<  \underline{\theta}$}
   {$t_{-}\leftarrow \tilde \rho$}
    Go to line \ref{ln:line_set}
}
{\Return $(\rho, x, y)$
           }
\end{algorithm}

\mgap
\mgap

Note that Algorithm \ref{alg_proc} is nothing but a bracketing/bisection procedure specialized to Problem A (see Proposition \ref{pr:roseli}, Lemma \ref{lm:just} and Figure \ref{fig01}).

\mgap
\mgap

Next we show that Algorithm \ref{alg_proc} finishes with $\rho>0$ and $x\in \HH$ solving Problem A. 

\bprop \lab{pr:bterminates}
Consider the setting of \emph{Problem A} above and assume that $w\notin (A+B+C+D)(Gz)$.
Then, the bracketing/bisection procedure, as presented in \emph{Algorithm \ref{alg_proc}}, terminates with a triple
$(\rho, x, y)$ with $\rho>0$ and $x\in \HH$ satisfying \eqref{eq:cha3}-\eqref{eq:cha4} and 
\[
y = (G z-x)/\rho + w + \left[B(x) - B(G z)\right] + \left[D (x) - D_{(G z)}(x)\right].
\]
\eprop
\bproo
The proof follows from Proposition \ref{pr:roseli} and Lemmas \ref{lm:mon.sva.gen02} and \ref{lm:just}. 
\eproo

\def\cprime{$'$} \def\cprime{$'$}

\end{document}